%% file: main.tex
\documentclass[11pt]{article}
\usepackage[percent]{overpic}
\usepackage{amsmath}
\usepackage{amssymb}
\usepackage{amsthm}
\usepackage{array} 
\usepackage[english]{babel}
\usepackage{enumitem}
\usepackage[margin=2.5cm]{geometry}
\usepackage{lmodern}
\usepackage{xfrac}  
\usepackage{soul}  

\usepackage{array} 
\usepackage{multirow}
\newcolumntype{C}[1]{>{\centering\arraybackslash}m{#1}} 

\usepackage{graphicx,float,wrapfig, caption}
\usepackage[percent]{overpic}
\usepackage{pgf, tikz, tikz-cd} 
\graphicspath{ {pictures/}{tikz/} }
\usepackage{mathtools} 
\usetikzlibrary{arrows, backgrounds} 
\usepackage{extarrows}
\usetikzlibrary{positioning}
\usepackage{standalone} 
\usepackage{subcaption} 

\usepackage{hyperref}
\hypersetup{
    unicode=false,          
    pdftoolbar=true,        
    pdfmenubar=true,        
    pdffitwindow=false,     
    pdfstartview={FitH},    
    pdftitle={},    
    pdfauthor={Eduard Duryev},     
    pdfsubject={Subject},   
    pdfcreator={Creator},   
    pdfproducer={Producer}, 
    pdfkeywords={Translation surface} {Billiard} {Square-tiled surface}  {Modular curve} {Eigenform} {Branched covering} {Many automorphisms} , 
    pdfnewwindow=true,      
    colorlinks=true,       
    linkcolor=blue,          
    citecolor=blue,        
    filecolor=magenta,      
    urlcolor=cyan           
}

\usepackage{thmtools}
\declaretheorem[numberwithin=section]{theorem}

\declaretheorem[sibling=theorem]{proposition}

\declaretheorem[sibling=theorem, style=remark]{remark}
\newtheorem*{theorem*}{Theorem}
\newtheorem*{question*}{Question}
\newtheorem*{conjecture*}{Conjecture}

\usepackage{color}

\newcommand{\Z}{\mathbb Z}

\newcommand{\R}{\mathbb R}
\newcommand{\C}{\mathbb C}

\renewcommand{\P}{\mathbb P}

\newcommand{\M}{\mathcal M}

\renewcommand{\H}{\mathcal H}

\newcommand{\Aut}{\operatorname{Aut}}

\def\lowcomma{_{\textstyle,}}

\title{Eigenforms of hyperelliptic curves with many automorphisms}
\author{Eduard Duryev, Leonid Monin}
\date{}

\newcommand{\Addresses}{{
  \bigskip
  \footnotesize

  E.~Duryev, \textsc{Universit\'e Paris 7 Diderot, Paris, France}\par\nopagebreak
  \textit{E-mail address}: \texttt{eduryev@gmail.com}

  \medskip

  L.~Monin, \textsc{Max Planck Institute for Mathematics in the Sciences, Leipzig, Germany}\par\nopagebreak
  \textit{E-mail address}: \texttt{leonid.monin@mis.mpg.de}
}}

\makeatletter
\newcommand{\subjclass}[2][1991]{%
  \let\@oldtitle\@title%
  \gdef\@title{\@oldtitle\footnotetext{#1 \emph{Mathematics subject classification.} #2}}%
}
\newcommand{\keywords}[1]{%
  \let\@@oldtitle\@title%
  \gdef\@title{\@@oldtitle\footnotetext{\emph{Key words and phrases.} #1.}}%
}
\makeatother


\begin{document}

\maketitle

\begin{abstract}
Given a pair of translation surfaces it is very difficult to determine whether they are supported on the same algebraic curve. In fact, there are very few examples of such pairs. In this note we present 
infinitely many examples of finite collections of translation surfaces supported on the same algebraic curve. 

The underlying curves are hyperelliptic curves with many automorphisms. For each curve, the automorphism of maximal order acts on the space of holomorphic 1-forms. We present a translation surface corresponding to each of the eigenforms of this action.
\end{abstract}

\tableofcontents

\section{Introduction}
Let $X \in \M_g$ be an algebraic curve and let $f \in \Aut(X)$ be its automorphism. Let $\Omega(X)$ be a $g$-dimensional space of holomorphic 1-forms on $X$. An automorphism $f \in \Aut(X)$ induces an action $f^*: \Omega(X) \to \Omega(X)$ via pullback. A non-zero holomorphic 1-form $\omega \in \Omega(X)$ is called an {\em eigenform of $f$} if it is an eigenvector of $f^*$.

Following the notation of \cite{R} and \cite{W} we say that an algebraic curve $X$ has {\em many automorphisms} if it does not belong to a non-trivial family of curves with the same group of automorphisms. All hyperelliptic curves with many automorphisms are listed in \cite[Table~1]{MP}. The list consists of three infinite families spread across all genera and sporadic examples. In this note we focus on the three infinite families.

A pair $(X, \omega)$, where $X \in \M_g$ is an algebraic curve, or a Riemann surface, of genus $g$ and $\omega$ is a non-zero holomorphic 1-form on $X$, is called an {\em Abelian differential}. The set of zeroes of $\omega$ is denoted by $Z(\omega)$. For any Abelian differential $(X,\omega)$ the integration of $\omega$ produces a {\em translation structure}, an atlas of complex charts on $X \setminus Z(\omega)$ with parallel translations $z\mapsto z+c$ as transition functions. A neighborhood of a conical point possesses a singular flat structure obtained by pulling back flat metric on a disk via the covering map $z \mapsto z^k$. Conversely, every translation structure defines an Abelian differential as pullback of complex structure from $\C$ and the holomorphic 1-form $dz$ on it. Abelian differentials are also called {\em translation surfaces}. In fact, every Abelian differential can be represented as a collection of polygons with pairs of parallel sides identified by translation (see Figure~\ref{fig:translation}). For more background on translation surfaces see \cite{Z}.

\begin{figure}[H]
\centering
\includegraphics[width=0.4\linewidth]{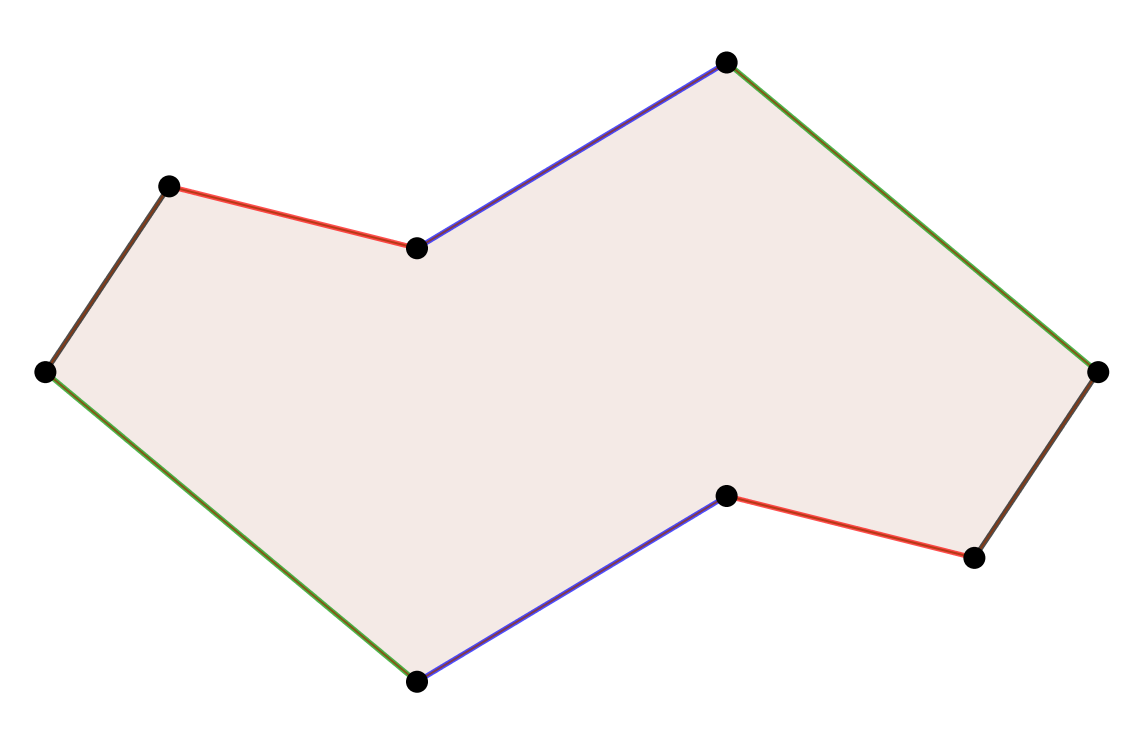}
\caption{An example of a translation surface: the edges of the polygon of the same color are identified by parallel translations.}
\label{fig:translation}
\end{figure}

This gives a convenient way of visualizing holomorphic 1-forms on algebraic curves. However it is almost never possible to determine the underlying algebraic curve from the polygonal picture. In particular, there are very few examples of pairs of polygonal representations of holomorphic 1-forms $\omega$ and $\omega'$ supported on the same algebraic curve $X$. We use eigenforms to provide such examples. In this note we construct translation surfaces representing eigenbases on hyperelliptic curves with many automorphisms:

\begin{theorem}\label{thm:main}
Let $S^k_N$ and $H^k_N$ be polygonal representations of translation surfaces defined in Section~\ref{sec:eigen}. Then for any $k = 1, \ldots, g$:
\begin{itemize}
    \item[1.] $S^k_{2g+2}$ represents an eigenform $\omega_k = x^{k-1} dx/y$ on the hyperelliptic curve 
    ${\mathcal A}_g: y^2=x^{2g+2}-1$;
    \item[2.] $S^k_{2g+1}$ represents an eigenform $\omega_k = x^{k-1} dx/y$ on the hyperelliptic curve
    ${\mathcal P}_g:y^2=x^{2g+1}-1$;
    \item[3.] $H^{k}_{2g}$ represents an eigenform $\omega_k = x^{k-1} dx/y$ on the hyperelliptic curve
    ${\mathcal D}_g:y^2=x(x^{2g}-1)$.
\end{itemize}
\end{theorem}

The similar construction for the polar hypereliptic curves of odd genus appeared in \cite{kucharczyk}.

\noindent
{\bf Acknowledgements.} 
This research has emerged from the question of Curt McMullen on representations of eigenforms as translation surfaces.

During this work, the first author was supported by FSMP postdoc grant, and the second author was supported by EPSRC Early Career Fellowship EP/R023379/1.

\section{Hyperelliptic curves with many automorphisms}
\label{sec:curves}

In this section we describe three families of hyperelliptic curves with many automorphisms and their eigenforms. The content of this section is summarized in Table~\ref{table1}. Consider the following curves in the affine plane:

\begin{align*}
   {\mathcal A}_g: & \ y^2 = x^{2g+2} - 1 \\
   {\mathcal P}_g: & \ y^2 = x^{2g+1} - 1 \\
   {\mathcal D}_g: & \ y^2 = x(x^{2g} - 1)\\
\end{align*}

\begin{figure}[H]
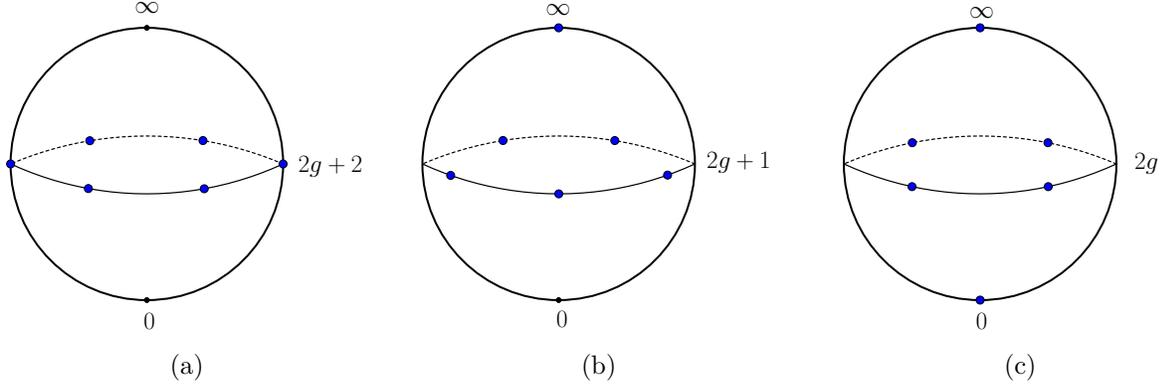

\begin{subfigure}{.33\textwidth}
  \centering
  \includestandalone[width=\linewidth]{apolar}
  \caption{}
  \label{fig:apolar}
\end{subfigure}%
\begin{subfigure}{.33\textwidth}
  \centering
  \includestandalone[width=\linewidth]{polar}
  \caption{}
  \label{fig:polar}
\end{subfigure}
\begin{subfigure}{.33\textwidth}
  \centering
  \includestandalone[width=\linewidth]{dipolar}
  \caption{}
  \label{fig:dipolar}
\end{subfigure}
\caption{The branching loci of hyperelliptic curves (a) ${\mathcal A}_g$, (b) ${\mathcal P}_g$ and (c) ${\mathcal D}_g$, via hyperelliptic covers to $\P^1$.}
\label{fig:branchloci}
\end{figure}

The homogenization of these equations define algebraic curves in $\P^2$. Curve ${\mathcal D}_g$ is smooth, however ${\mathcal A}_g$ and ${\mathcal P}_g$ have singularities at $[0:0:1]$, and one needs to consider normalizations of these curves.
\begin{remark}
Apart from these three families there are only finitely many hyperelliptic curves with many automorphisms (see \cite{MP}).
\end{remark}

The hyperelliptic involutions of ${\mathcal A}_g, {\mathcal P}_g, {\mathcal D}_g$ are given by $\eta: (x,y) \mapsto (x,-y)$. For any hyperelliptic curve $X \in \M_g$ the hyperelliptic cover $\pi_\eta: X \to \P^1$ of degree two is  branched over $2g+2$ points on $\P^1$. These points uniquely determine the curve $X$ (see Figure~\ref{fig:branchloci}). 
We describe the three above families in terms of their hyperelliptic covers and exhibit their automorphisms of locally maximal order. Let $\xi_k$ be a primitive $k$-th root of unity.

\noindent
{\bf Apolar curve.}
The hyperelliptic curve ${\mathcal A}_g$ is a double cover over $\P^1$ branched over $(2g+2)$-th roots of unity (see Figure~\ref{fig:apolar}). It has an automorphism of order $2g+2$ given by:
$$
f_{{\mathcal A}_g}: (x,y) \mapsto (\xi_{2g+2} x , y).
$$
We will call ${\mathcal A}_g$ an {\em apolar curve}.

\noindent
{\bf Polar curve.}
The hyperelliptic curve ${\mathcal P}_g$ is a double cover over $\P^1$ branched over $\infty$ and $(2g+1)$-st roots of unity (see Figure~\ref{fig:polar}). An automorphism of order $2g+1$ is given by:
$$
f_{{\mathcal P}_g}: (x,y) \mapsto (\xi_{2g+1} x , y).
$$
We will call ${\mathcal P}_g$ a {\em polar curve}.

\noindent
{\bf Dipolar curve.}
The hyperelliptic curve ${\mathcal D}_g$ is a double cover over $\P^1$ branched over $0$, $\infty$ and $2g$-th roots of unity (see Figure~\ref{fig:dipolar}). An automorphism of order $4g$ is given by:
$$
f_{{\mathcal D}_g}: (x,y) \mapsto (\xi_{2g} x , \xi_{4g} y).
$$
We will call ${\mathcal D}_g$ a {\em dipolar curve}.

\noindent
{\bf Eigenbasis.}
For any $g>1$ the forms:
$$
\omega_k = x^{k-1} dx/y, \ k = 1,\ldots,g
$$
define a basis of 1-forms on ${\mathcal A}_g$, ${\mathcal P}_g$ and ${\mathcal D}_g$. The induced action of $f_{{\mathcal A}_g}, f_{{\mathcal P}_g}$ and  $f_{{\mathcal D}_g}$ on these forms is given by:
\begin{equation} \label{eigenvalueA}
f_{{\mathcal A}_g}^*(\omega_k) = \xi_{2g+2}^{k} \cdot \omega_k
\end{equation}
\begin{equation} \label{eigenvalueP}
f_{{\mathcal P}_g}^*(\omega_k) = \xi_{2g+1}^{k} \cdot \omega_k
\end{equation}
\begin{equation} \label{eigenvalueB}
f_{{\mathcal D}_g}^*(\omega_k) =\xi_{4g}^{2k-1} \cdot \omega_k
\end{equation}

Therefore $\{ \omega_k \mid k=1,\ldots,g\}$ is an eigenbasis for the automorphisms $f_{{\mathcal A}_g}$, $f_{{\mathcal P}_g}$ and $f_{{\mathcal D}_g}$. Next we find the multiplicities of the zeroes of these eigenforms, in other words determine the strata of Abelian differentials that contain $({\mathcal A}_g, \omega_k)$, $({\mathcal P}_g, \omega_k)$ and $({\mathcal D}_g, \omega_k)$.

\noindent
{\bf Strata.} 
The eigenform $\omega_k$ on ${\mathcal A}_g$ has two zeroes of order $k-1$ at the preimages of $0 \in \P^1$ and two zeroes of order $g-k$ at the preimages of $\infty$. Indeed, the hyperelliptic cover is unbranched over $0$ and $\infty$. Therefore for $(x=0,y=\pm i)$ one can use $x$ as a coordinate and $x^{k-1} dx/y$ has zero of order $k-1$ at each of these points. The only other zeroes of $\omega_k$ are in the fiber over $\infty$, they are switched by hyperelliptic involution, and therefore they have the same order. Considering that the sum of orders of all zeroes of $\omega_k$ is $2g-2$, each of the zeroes over $\infty$ has order $g-k$. Therefore, $({\mathcal A}_g, \omega_k) \in \H(k-1,k-1,g-k,g-k)$.

The eigenform $\omega_k$ on ${\mathcal P}_g$ has a zero of order $2g-2k$ over $\infty \in \P^1$ and two zeroes of order $k-1$ over $0 \in \P^1$. Indeed, the hyperelliptic cover is unbranched over $0$ and hence $\omega_k$ has two zeroes of orders $k-1$ at the preimages. At the same time the cover has only one preimage over $\infty$, and since the sum of orders of all zeroes of $\omega_k$ is $2g-2$, this preimage is a zero of order $2g-2k$. Therefore $({\mathcal P}_g, \omega_k) \in \H(k-1,k-1,2g-2k)$.

Similarly, $({\mathcal D}_g,\omega_k) \in \H(2k-2,2g-2k)$.

\begin{table}[H]
\centering
\begin{tabular}{ | C{0.5cm} | C{2.7cm} | C{3.5cm} | C{2cm} | C{3.5cm} | C{1.7cm} | }
\hline
& & & & &\\
 & Equation & Automorphism $f$ & Eigenform $\omega_k$ & Stratum & $f^*(\omega_k)$  \\ 
& & & & &\\
\hline
& & & & &\\
${\mathcal A}_g$ & $y^2=x^{2g+2}-1$ & $(x,y) \mapsto (\xi_{2g+2} x , y)$ & $\displaystyle \frac{x^{k-1}dx}y$ & $\H((k-1)^2,(g-k)^2)$ & $\xi_{2g+2}^{k} \cdot \omega_k$   \\ 
& & & & &\\
\hline
& & & & &\\
${\mathcal P}_g$ & $y^2=x^{2g+1}-1$ & $(x,y) \mapsto (\xi_{2g+1} x , y)$ & $\displaystyle\frac{x^{k-1}dx}y$& $\H((k-1)^2,2g-2k)$ & $\xi_{2g+1}^{k} \cdot \omega_k$ \\  
& & & & &\\
\hline
& & & & &\\
${\mathcal D}_g$ & $y^2=x(x^{2g}-1)$ & $(x,y) \mapsto (\xi_{2g} x , \xi_{4g}y)$ & $\displaystyle\frac{x^{k-1}dx}y$ & $\H(2k-2,2g-2k)$ & $\xi_{4g}^{2k-1} \cdot \omega_k$ \\
& & & & &\\
\hline

\end{tabular}
\caption{Hyperelliptic curves with many automorphisms and their eigenforms.}
\label{table1}
\end{table}


\section{Translation surfaces and eigenforms}
\label{sec:eigen}
In this section we construct polygonal representations of translation surfaces corresponding to all eigenforms of hyperelliptic curves ${\mathcal A}_g, {\mathcal P}_g$ and ${\mathcal D}_g$. The content of this section is summarized in Table~\ref{table2}.

\noindent
{\bf Spingon.}
For any integers $N>1$, $1 \le k < (N+1)/2$ we define a spingon $S^k_N$ as follows. First, consider a rhombus $R_0 = OA_0B_0C_0$ with an angle $2\pi k/N$ at the vertex $O$. Let $r: \R^2 \to \R^2$ be the rotation by $2\pi k/N$ with the center at $O$. The rotation $r$ generates a cyclic subgroup of isometries of $\R^2$ isomorphic to $\Z/N\Z$. For any $i \in \Z/N\Z$ define $R_i = OA_iB_iC_i$ to be an image of $R_0$ under $r^i$. Now identify the following pairs of parallel sides by parallel translations:
$$
OC_i \sim OA_{i+1} \mbox{ and } A_iB_i \sim C_{i+1}B_{i+1}, \mbox{ for all } i \in \Z/N\Z.
$$
The resulting polygon with the above identifications of parallel sides is a translation surface that we will call a {\em spingon} and will denote by $S^k_N$. On  Figure~\ref{fig:spingons} we provide examples of spingons in genus~3.

\begin{figure}[H]
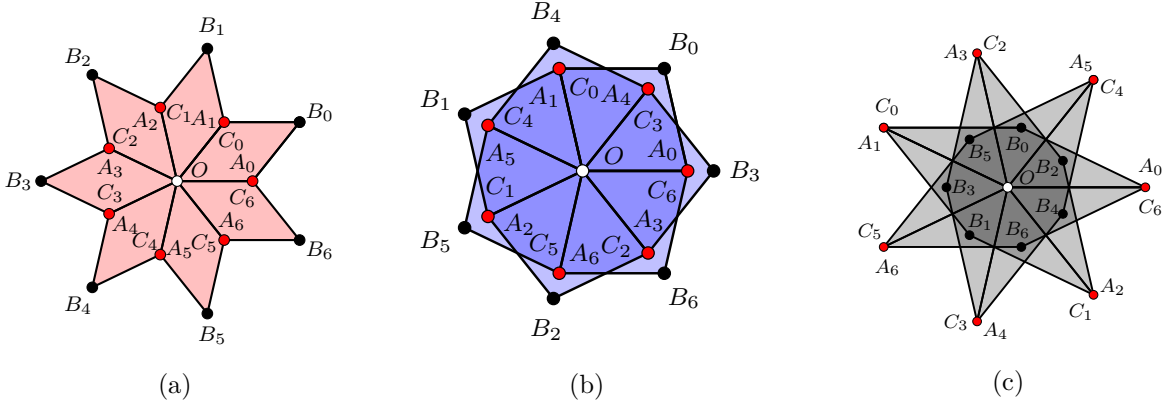

\begin{subfigure}{.33\textwidth}
  \centering
  \includestandalone[width=\linewidth]{S71}
  \caption{}
  \label{fig:spingon(1,7)}
\end{subfigure}%
\begin{subfigure}{.33\textwidth}
  \centering
  \includestandalone[width=\linewidth]{S72}
  \caption{}
  \label{fig:spingon(2,7)}
\end{subfigure}
\begin{subfigure}{.33\textwidth}
  \centering
  \includestandalone[width=\linewidth]{S73}
  \caption{}
  \label{fig:spingon(3,7)}
\end{subfigure}
\caption{The spingons (a) $S^1_7$, (b) $S^2_7$ and (c) $S^3_7$.}
\label{fig:spingons}
\end{figure}

\noindent
{\bf Half-spingon.} For any integers $N>1$ and $1 \le k \le N/2$ we define half-spingon $H^k_N$ as follows. Start with the rhombus $P_0 = OA_0B_0C_0$ with an angle $\pi (2k-1)/N$ at the vertex $O$. A union of images of $P_0$ under $N-1$ successive rotations around $O$ by an angle $\pi (2k-1)/N$ is a polygon that consists of the rhombi $P_i$. This time, because $2k-1$ is odd, the sides $OA_0$ and $OC_{N-1}$ are given by the opposite vectors. The identifications are:
\begin{align*}
OA_0 & \sim B_{N-1}A_{N-1};\ C_0B_0\sim C_{N-1}O;\\
OC_i & \sim OA_{i+1}, \mbox{ for } i=0,\ldots,N-2;\\
A_iB_I & \sim C_{i+1}B_{i+1}, \mbox{ for } i=0,\ldots,N-2.
\end{align*}
The resulting polygon with these identifications is a translation surface that we will call a {\em half-spingon} and will denote it by $H^k_N$. On Figure~\ref{fig:half-spingons} we provide examples of half-spingons in genus 3. We only label vertices on figure (a), the reader can reconstruct the identifications for the other two figures.

\begin{figure}[H]
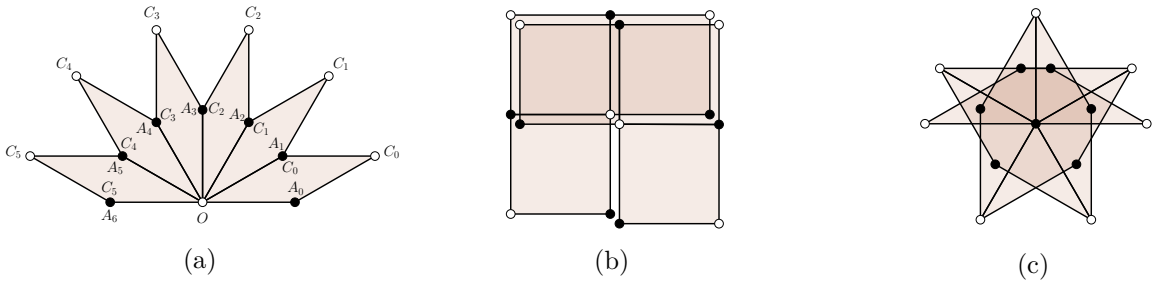

\begin{subfigure}{.33\textwidth}
  \centering
  \includestandalone[width=\linewidth]{H61}
  \caption{}
  \label{fig:half-spingon(1,6)}
\end{subfigure}%
\begin{subfigure}{.33\textwidth}
  \centering
  \includestandalone[width=\linewidth]{H62}
  \caption{}
  \label{fig:half-spingon(2,6)}
\end{subfigure}
\begin{subfigure}{.33\textwidth}
  \centering
  \includestandalone[width=\linewidth]{H63}
  \caption{}
  \label{fig:half-spingon(3,6)}
\end{subfigure}
\caption{The half-spingons (a) $H^1_6$, (b) $H^2_6$ and (c) $H^3_6$.}
\label{fig:half-spingons}
\end{figure}

Now we are ready to prove Theorem~\ref{thm:main}.

\begin{proof}[Proof of Theorem~\ref{thm:main}]

\noindent
{\bf Apolar curves ${\mathcal A}_g$.} Let $g>1$ be any integer. We will show that $S^k_{2g+2}$ represents an eigenform $\omega_k = x^{k-1} dx/y$ on the apolar hyperelliptic curve ${\mathcal A}_g$.

Note that all vertices $B_i$ are identified into one that we denote by $B$. Since the number of rhombi $2g+2$ is even, all $A_i$ and $C_i$ are identified into two other vertices that we denote by $A$ and $C$. Hence there are four vertices: $O$, $A$, $B$ and $C$.

We first identify which stratum does translation surface $S^k_{2g+2}$ belongs. The total angle around $O$ and $B$ is $2\pi k$, and around $A$ and $C$ it is $(2g+2)(\pi - 2\pi k/(2g+2)) = 2\pi(g-k+1)$. Hence $S^k_{2g+2}$ belongs to $\H(k-1,k-1,g-k,g-k)$. In particular, the genus of the underlying algebraic curve of $S^k_{2g+2}$ is indeed $g$.

Now let us show that the underlying Riemann surface of $S^k_{2g+2}$ is hyperelliptic. The hyperelliptic involution is given by rotating each rhombus $R_i = OA_iB_iC_i$ around its center by $180^\circ$. Because of the side identifications this is a well-defined map from $S^k_{2g+2}$ to itself. It has $2g+2$ fixed points: one at the center of each of the $2g+2$ rhombi. This implies that $S^k_{2g+2}$ is hyperelliptic.

The automorphism $f_{{\mathcal A}_g}$ of order $2g+2$ is given by rotating $S^k_{2g+2}$ around $O$ by the angle $2\pi k /(2g+2)$ sending each $R_i$ to $R_{i+1}$. Since hyperelliptic curve of genus $g$ with an automorphism of order $2g+2$ is unique we conclude that $S^k_{2g+2}$ is supported on the apolar curve ${\mathcal A}_g$. Finally $S^k_{2g+2} \in \H(k-1,k-1,g-k,g-k)$ is preserved up to a scalar by $f$
and we conclude that $S^k_{2g+2} = ({\mathcal A}_g, \omega_k)$. Note that the eigenvalue of the automorphism $f_{{\mathcal A}_g}$ is $\xi_{2g+2}^{k}$, which agrees with~(\ref{eigenvalueA}).

\noindent
{\bf Polar curves ${\mathcal P}_g$.} Next we will show that $S^k_{2g+1}$ represents an eigenform $\omega_k = x^{k-1} dx/y$ on the polar hyperelliptic curve ${\mathcal P}_g$.

In this case all vertices $B_i$ are identified into one that we denote by $B$, and all $A_i$ and $C_i$ are identified into another that we denote by $A$. Hence there are three vertices: $O$, $A$ and $B$ with total angles $2\pi k$, $2(2g+1)(\pi - 2\pi k/(2g+1))= 2\pi(2g-2k+1)$ and $2\pi k$ respectively. Hence $S^k_{2g+1}$ belongs to $\H(k-1,k-1,2g-2k)$. In particular, the genus of the underlying algebraic curve of $S^k_{2g+1}$ is indeed $g$.

Similarly the hyperelliptic involution is given by rotating each rhombus $R_i = OA_iB_iC_i$ around its center by $180^\circ$. It has $2g+2$ fixed points: the point $A$ and the center of each of the $2g+1$ rhombi. This implies that $S^k_{2g+1}$ is hyperelliptic.

The automorphism $f_{{\mathcal P}_g}$ of order $2g+1$ is given by the rotation around $O$ by $2\pi k /(2g+1)$. 
Since the hyperelliptic curve of genus $g$ with an automorphism of order $2g+1$ is unique, the 1-form of $S^k_{2g+1} \in \H(k-1,k-1,2g-2k)$ is preserved by $f_{{\mathcal P}_g}$ up to scalar, we conclude that $S^k_{2g+1} = ({\mathcal P}_g, \omega_k)$. Again, the eigenvalue of the automorphism $f_{{\mathcal P}_g}$ is $\xi^{k}_{2g+1}$, which agrees with~(\ref{eigenvalueP}).

\noindent
{\bf Dipolar curves ${\mathcal D}_g$.}
Next we will show that $H^{k}_{2g}$ represents an eigenform $\omega_k = x^{k-1} dx/y$ on the polar hyperelliptic curve ${\mathcal D}_g$.

The vertices $B_i$ and $O$ are identified into one that we denote by $O$, and the vertices $A_i$ and $C_i$ are identified into another that we denote by $A$. Hence there are two vertices: $O$ and $A$ with total angles $2\pi (2k-1)$ and  $2\cdot2g(\pi - 2\pi(2k-1)/4g)= 2\pi(2g-2k+1)$ respectively. Hence $H^{k}_{2g}$ belongs to $\H(2k-2,2g-2k)$. In particular, the genus of the underlying algebraic curve of $H^{k}_{2g}$ is indeed $g$.

The hyperelliptic involution is given by rotating each rhombus $P_i = OA_iB_iC_i$ around its center by $180^\circ$. It has $2g+2$ fixed points: the points $O$ and $A$, and a point at the center of each of the $2g$ rhombi. This implies that $H^{k}_{2g}$ is hyperelliptic.

The automorphism $f_{{\mathcal D}_g}$ of order $4g$ is given by sending each $R_i$ to $R_{i+1}$ via rotation around $O$ by $\pi (2k-1) /2g$. Since the hyperelliptic curve of genus $g$ with an automorphism of order $4g$ is unique, the 1-form of $H^{k}_{2g} \in \H(2k-2,2g-2k)$ is preserved by $f_{{\mathcal D}_g}$ up to scalar, we conclude that $H^{k}_{2g} = ({\mathcal D}_g, \omega_k)$. Finally, the eigenvalue of the automorphism $f_{{\mathcal D}_g}$ is $\xi_{4g}^{2k-1}$, which agrees with~(\ref{eigenvalueP}).
\end{proof}

\begin{remark}
\label{remark}
Note that apolar and dipolar curves admit an extra automorphism $\iota: (x,y) \mapsto (1/x, iy/x^{g+1})$.  Notice that $\iota^2$ is the hyperelliptic involution. Since $\iota$ commutes with hyperelliptic involution it descends to the sphere and acts there by switching the North and the South poles. The action of $\iota$ on eigenforms is given by $\iota^*(\omega_k) = i \cdot \omega_{g-k+1}$. Therefore translation surfaces represented by $S_{2g+2}^k$ and $S_{2g+2}^{g-k+1}$ (similarly, $H^{k}_{2g}$ and $H^{2(g-k+1)}_{2g}$) are the same up to rotation. This can be observed in the construction of spingons (and half-spingons) by switching the center of rotation from $O$ to $A$.
\end{remark}

\noindent
{\bf Familiar examples.}
The collection of described eigenforms contains familiar families of examples of symmetric translation surfaces like regular $n$-gons and double $n$-gons. Below we give a more specific correspondence.

Recall that a {\em regular $n$-gon} is a translation surface obtained by identifying opposite sides of a regular $n$-gon via parallel translations. A {\em double $n$-gon} is a translation surface obtained by taking a regular $n$-gon together with its copy rotated by $\pi$, and identifying the sides with their rotational images via parallel translations. Below we identify, which eigenforms and spingons represent regular and double $n$-gons.

\begin{proposition}
\label{prop:correspondance}
The following identifications hold:

\begin{enumerate}
\item[(A)] Double $(2g+2)$-gon coincides with the eigenforms $({\mathcal A}_g,\omega_1) = ({\mathcal A}_g, \omega_g)$ also represented by the spingons $S_{2g+2}^1 = S_{2g+2}^{g}$ up to a complex scale factor.

\item[(B)] Double $(2g+1)$-gon coincides with the eigenform $({\mathcal P}_g,\omega_1)$ also represented by the spingon $S_{2g+1}^{1}$ up to a complex scale factor.

\item[(C)] Regular $(4g+2)$-gon coincides with the eigenform $({\mathcal P}_g,\omega_g)$ also represented by the spingon $S_{2g+1}^{g}$ up to a complex scale factor.

\item[(D)] Regular $(4g)$-gon coincides with the eigenforms $({\mathcal D}_g, \omega_1) = ({\mathcal D}_g, \omega_g)$ also represented by the half-spingons $H_{2g}^1 = H_{2g}^{g}$ up to a complex scale factor.
\end{enumerate}
\end{proposition}

We leave the proof to the reader. On Figures~\ref{fig:doublegon}, \ref{fig:regular14gon} and \ref{fig:regular12gon} we demonstrate how to cut and reglue regular and double $n$-gons to obtain spingons or half-spingons.

\begin{figure}[H]
\centering
\includegraphics[width=0.6\linewidth]{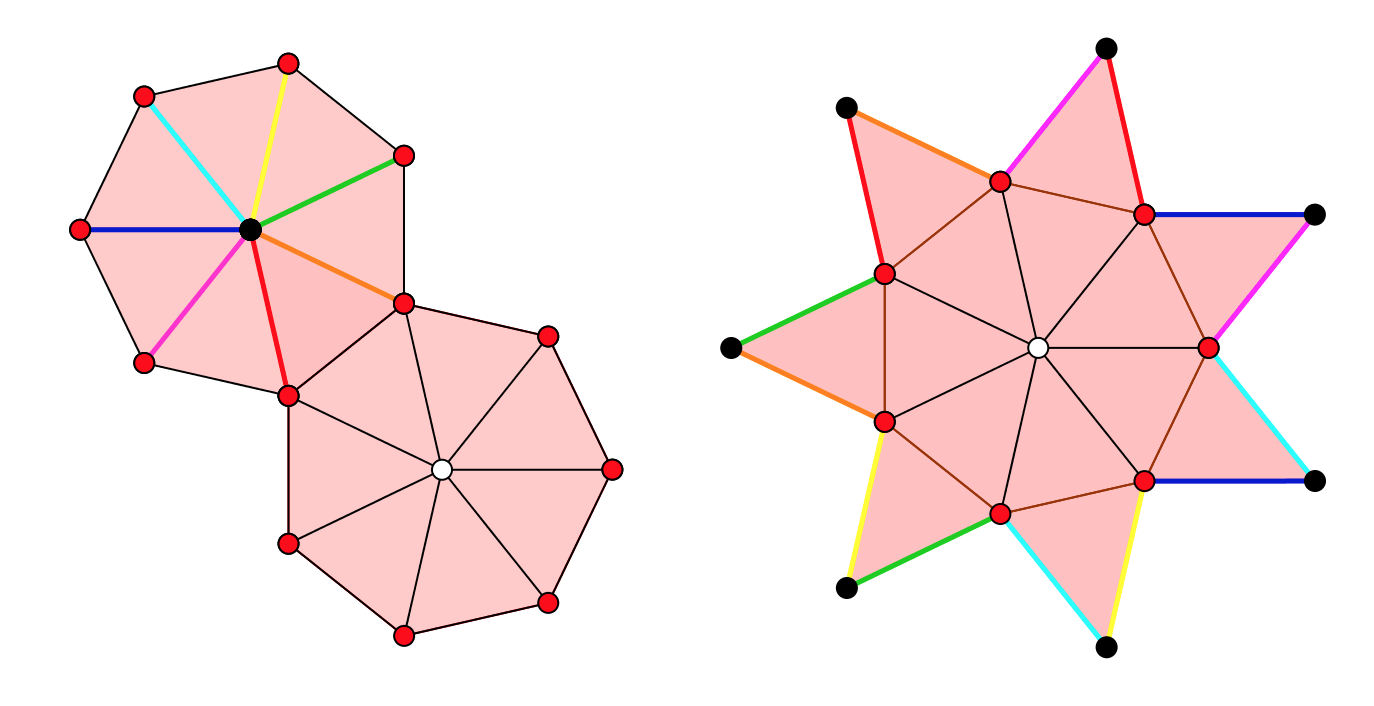}
\caption{Double 7-gon is the same as the spingon $S_7^1$. See Proposition~\ref{prop:correspondance}~(B).}
\label{fig:doublegon}
\end{figure}

\begin{figure}[H]
\centering
\includegraphics[width=0.6\linewidth]{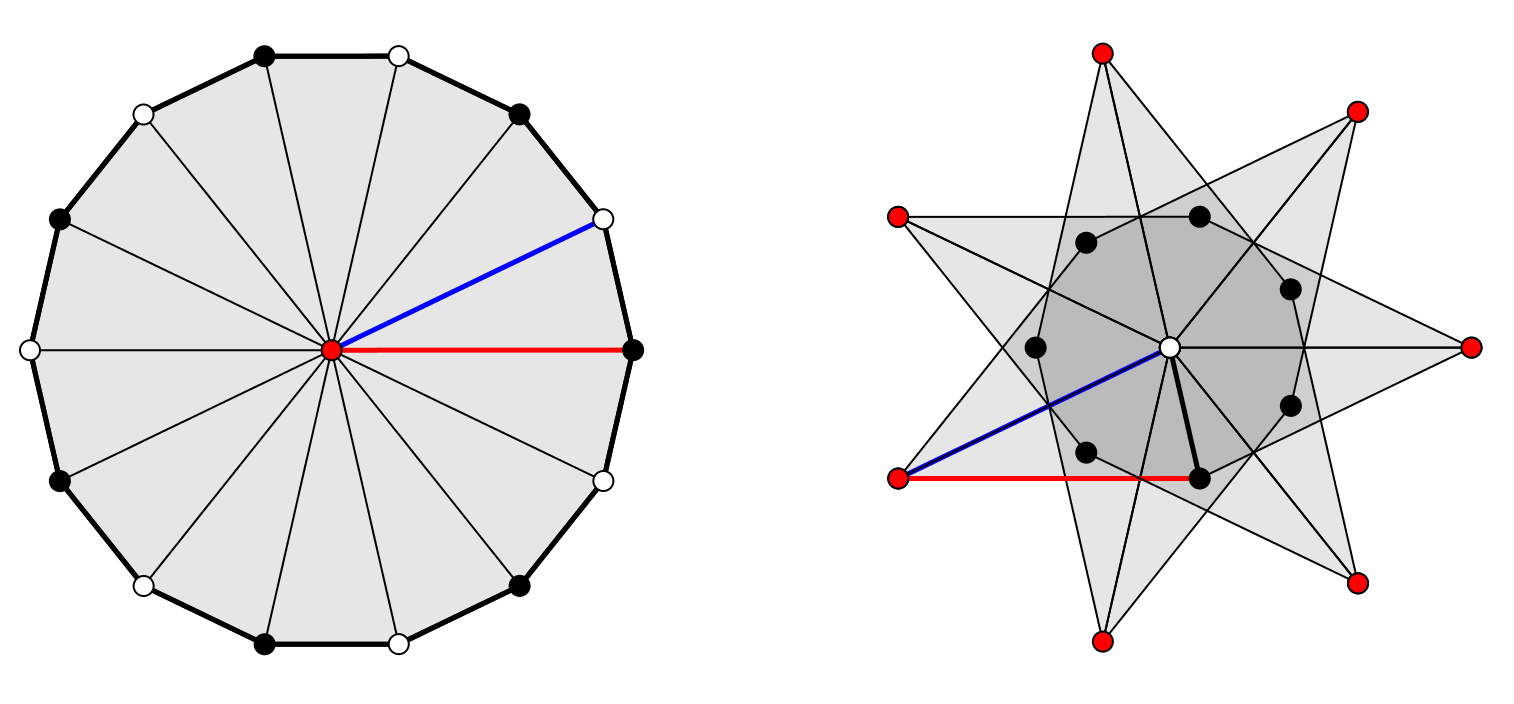}
\caption{Regular 14-gon is the same as the spingon $S_7^3$. See Proposition~\ref{prop:correspondance}~(C).}
\label{fig:regular14gon}
\end{figure}

\begin{figure}[H]
\centering
\includegraphics[width=0.6\linewidth]{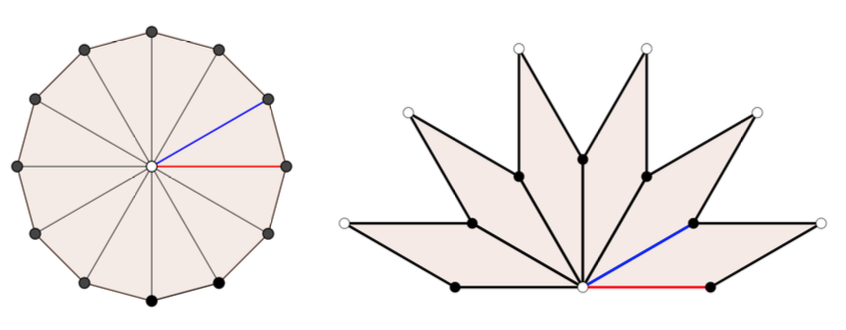}
\caption{Regular 12-gon is the same as the half-spingons $H_6^1=H_6^3$. See Proposition~\ref{prop:correspondance}~(D).}
\label{fig:regular12gon}
\end{figure}

\noindent
{\bf Examples in low genus.}
In Tables~\ref{fig:eigenforms2} and~\ref{fig:eigenforms3} we present more traditional polygonal representations of all eigenforms on hyperelliptic curves $\mathcal{A}_g$, $\mathcal{P}_g$ and $\mathcal{D}_g$ in genus $2$ and $3$ respectively. Note that when identifications of parallel edges are ambiguous or not obvious, the identified edges are labeled by the same color. One can observe that a lot of these eigenforms come from regular $n$-gons and double $n$-gons, but not all of them. For example, $(\mathcal{P}_3, \omega_2)$ is new.

Another interesting observation is that $(\mathcal{D}_{2k-1}, \omega_k)$ are particular examples of {\em square-tiled cyclic covers} introduced in \cite{FMZ}. We refer the reader to this paper for the background and definition. The equations of $\mathcal{D}_{2k-1}$ can be rewritten in the following form: $y^{2g} = (x-1)(x+1)(x-i)^{g}(x+i)^g$ (compare to \cite[Section~2.1]{FMZ}). The automorphism that realizes the corresponding cyclic cover over the Riemann sphere is given by a composition of translation by vector $(1,1)$ and rotation by $\pi$. Although this does not work for $(\mathcal{A}_{2k-1}, \omega_k)$, the quotient by the similar automorphism is a cyclic cover over a torus. 

\begin{table}[H]
  \centering
  \begin{tabular}{ | c | >{\centering\arraybackslash}p{7cm} | >{\centering\arraybackslash}p{7cm} | }
    \hline
    & $\omega_1$ & $\omega_2$ \\ \hline
    \multirow{2}{0.5cm}{$\mathcal{A}_2$}
    &
    \multicolumn{2}{|c|}{
    \begin{minipage}{.3\textwidth}
    \centering
    \includegraphics[height = 2cm]{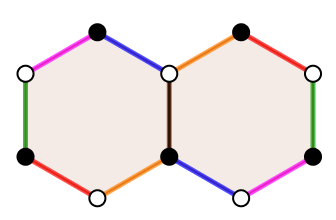}
    \end{minipage}} \\
    &    
    \multicolumn{2}{|c|}{$S_6^1 = S_6^2$}
    \\
    \hline
    \multirow{2}{0.5cm}{$\mathcal{P}_2$}
    &    
    \begin{minipage}{.3\textwidth}
    \centering
    \includegraphics[height = 1.75cm]{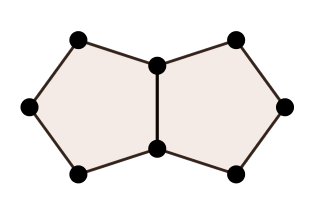}
    \end{minipage}
    &
    \begin{minipage}{.3\textwidth}
    \centering
    \includegraphics[height = 2cm]{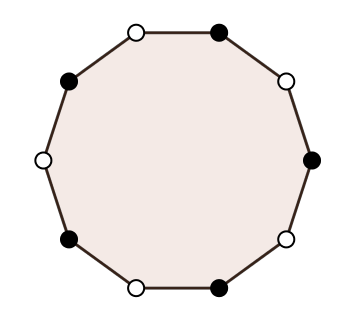}
    \end{minipage} \\
    &
    $S_5^1$
    &        
    $S_5^2$
    \\ \hline
    \multirow{2}{0.5cm}{$\mathcal{D}_2$}
    &
    \multicolumn{2}{|c|}{
    \begin{minipage}{.3\textwidth}
    \centering
    \includegraphics[height = 2cm]{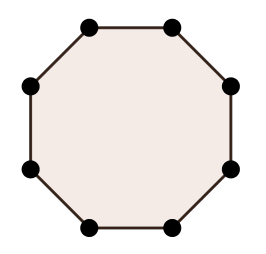}
    \end{minipage}} \\
    &
    \multicolumn{2}{|c|}{$H_4^1 = H_4^2$}
    \\ \hline
  \end{tabular}
  \caption{Eigenforms on hyperelliptic curves with many automorphisms in genus $2$.}
  \label{fig:eigenforms2}
\end{table}

\begin{table}[H]
  \centering
  \begin{tabular}{ | c | >{\centering\arraybackslash}p{4.5cm} | >{\centering\arraybackslash}p{4.5cm} | >{\centering\arraybackslash}p{4.5cm} | }
    \hline
    & $\omega_1$ & $\omega_3$ & $\omega_2$ 
    \\
    \hline
    \multirow{2}{0.5cm}{$\mathcal{A}_3$}
    &
    \multicolumn{2}{|c|}{
    \begin{minipage}{.3\textwidth}
    \centering
    \includegraphics[height = 2cm]{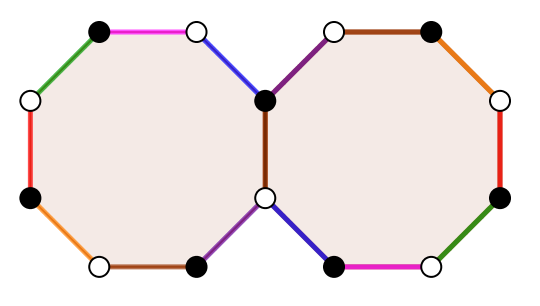}
    \end{minipage}}
    &
    \begin{minipage}{.25\textwidth}
    \centering
    \includegraphics[height = 2cm]{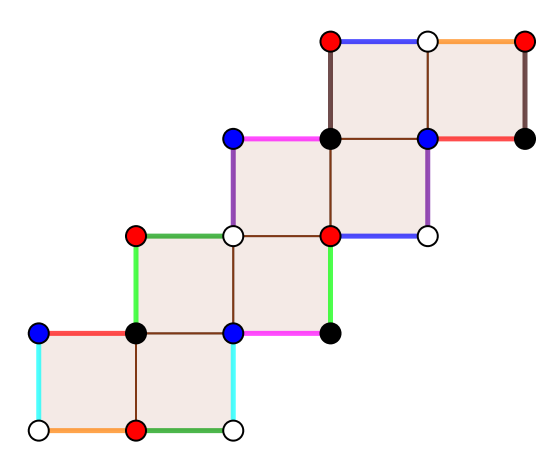}
    \end{minipage}
    \\
    &
    \multicolumn{2}{|c|}{$S_8^1 = S_8^3$} 
    &    
    $S_8^2$
    \\
    \hline
    $\mathcal{P}_3$
    &    
    \begin{minipage}{.25\textwidth}
    \centering
    \includegraphics[height = 2cm]{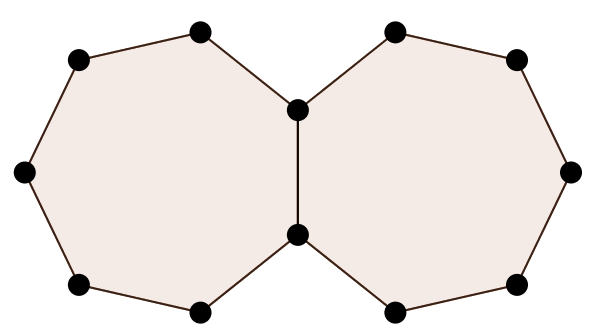}
    \end{minipage}
    $S_7^1$
    &    
    \begin{minipage}{.25\textwidth}
    \centering
    \includegraphics[height = 1.9cm]{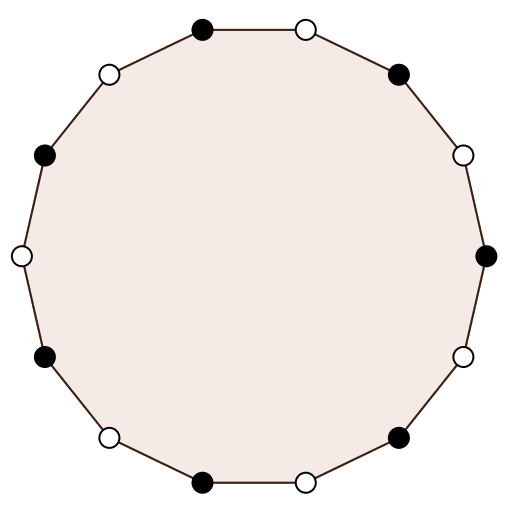}
    \end{minipage}
    $S_7^3$
    &    
    \begin{minipage}{.25\textwidth}
    \centering
    \includegraphics[height = 1.9cm]{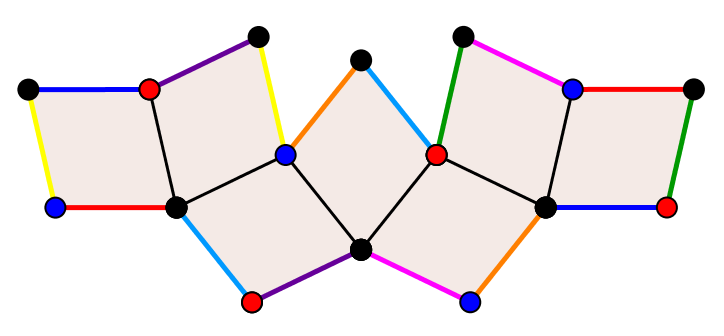}
    \end{minipage}
    $S_7^2$
    \\ \hline
    \multirow{2}{0.5cm}{$\mathcal{D}_3$}
    &
    \multicolumn{2}{|c|}{
    \begin{minipage}{.3\textwidth}
    \centering
    \includegraphics[height = 2cm]{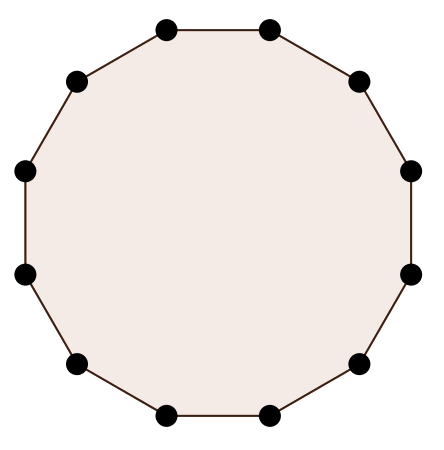}
    \end{minipage}}
    &
    \begin{minipage}{.25\textwidth}
    \centering
    \includegraphics[height = 2cm]{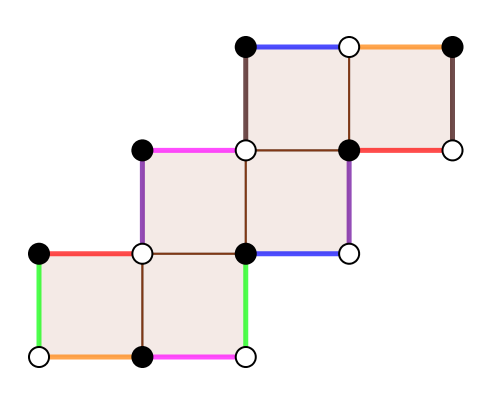}
    \end{minipage}
    \\
    &
    \multicolumn{2}{|c|}{$H_6^1 = H_6^3$}
    &    
    $H_6^2$ 
    \\ \hline
  \end{tabular}
  \caption{Eigenforms on hyperelliptic curves with many automorphisms in genus $3$.}
  \label{fig:eigenforms3}
\end{table}

\begin{table}[H]
\centering
\begin{tabular}{ | C{1.5cm} | C{4cm} | C{4cm} | C{5cm} | }
\hline
& & & \\
& Polygon & Rhombus angle $\angle O$ & Conical points \\ 
& & & \\
\hline
& & & \\
$({\mathcal A}_g,\omega_k)$  & $S^{k}_{2g+2}$ & $\displaystyle \frac{2\pi k}{2g+2}$ & \begin{tabular}{@{}c@{}} $O: 2\pi k, B: 2\pi k$ \\ $A: 2\pi(g-k+1)$ \\ $C: 2\pi(g-k+1)$ \end{tabular} \\ 
& & &\\
\hline
& & & \\
$({\mathcal P}_g,\omega_k)$ & $S^{k}_{2g+1}$ & $\displaystyle\frac{2\pi k}{2g+1}$ & \begin{tabular}{@{}c@{}} $O: 2\pi k, B: 2\pi k$ \\ $A=C: 2\pi(2g-2k+1)$ \end{tabular} \\  
& & & \\
\hline
& & & \\
$({\mathcal D}_g,\omega_k)$  & $H^{k}_{2g}$ & $\displaystyle\frac{\pi(2k-1)}{2g}$ & \begin{tabular}{@{}c@{}} $O=B: 2\pi(2k-1)$ \\ $A=C: 2\pi(2g-2k+1)$ \end{tabular} \\
& & & \\
\hline
\end{tabular}
\caption{Eigenforms are translation surfaces obtained by identifying parallel sides of polygons composed of equal rhombi $R_i=OA_iB_iC_i$.}
\label{table2}
\end{table}

\section{Triangulations and rational billiards}
The triangulation of each of the rhombi $R_i = OA_iB_iC_i$ obtained by adding the diagonals gives triangulations of spingons and half-spingons. Such triangulation is invariant under all automorphisms of the curve. Therefore each eigenform $\omega_k$ on $\mathcal{A}_g$, $\mathcal{P}_g$ and $\mathcal{D}_g$ uniquely determines a triangulation of the underlying Riemann surface. In this section we investigate these triangulations.

\begin{figure}[h]
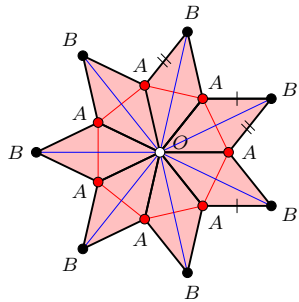

\centering
\includestandalone[width=0.3\linewidth]{S71_v1}
\caption{Triangulation $\mathcal{T}$ of $\mathcal{P}_3$ given by the eigenform $\omega_1$.}
\label{fig:S71_v1}
\end{figure}

We begin by setting some notation. Let $X$ be $\mathcal{A}_g$, $\mathcal{P}_g$ or $\mathcal{D}_g$ and let $N$ be $2g+2$, $2g+1$ or $2g$ respectively. Let $\omega_k$ be the eigenform $x^{k-1}dx/y$. Recall that $\eta: X \to X$ denotes the hyperelliptic involution and $\pi_\eta: X \to \P^1$ is the hyperelliptic cover given by $(x,y)\mapsto x$. The content of this section is summarized by the following proposition:

\begin{proposition}
Consider the triangulation $\mathcal{T}$ of $X$ given by the subdivision of each rhombus $R_i = OA_iB_iC_i$ of $(X,\omega_k)$ by its diagonals. Then:
\begin{enumerate}
    \item[1)] the triangulation $\mathcal{T}$ maps under the hyperelliptic cover to the triangulation of the sphere $\P^1$, which in the complex coordinate has the following vertices: $O$ at $0$, $A$ at $\infty$ and $W_i$'s at all $N$-th roots of unity, and the edges belonging to the radial rays passing through all $2N$-th roots of unity (see Figure~\ref{fig:sphere_triang1});
    \item[2)] the triangulation $\mathcal{T}$ does not depend on the choice of the eigenform $\omega_k$;
    \item[3)] the translation surface $(X,\omega_k)$ is an unfolding of the triangular billiard table with angles:
\begin{align}
    \label{angles1}
    {\mathcal A}_g: & \ \left(\frac{\pi}2 \lowcomma \frac{\pi k}{2g+2} \lowcomma \frac{\pi(g-k+1)}{2g+2}\right)\\
    \label{angles2}
    {\mathcal P}_g: & \left(\frac{\pi}2 \lowcomma \frac{\pi k}{2g+1} \lowcomma \frac{\pi(2g-2k+1)}{4g+2}\right) \\
    \label{angles3}
    {\mathcal D}_g: & \left(\frac{\pi}2 \lowcomma \frac{\pi(2k-1)}{4g} \lowcomma \frac{\pi(2g-2k+1)}{4g}\right).
\end{align}
\end{enumerate}
\end{proposition}

\begin{figure}[H]
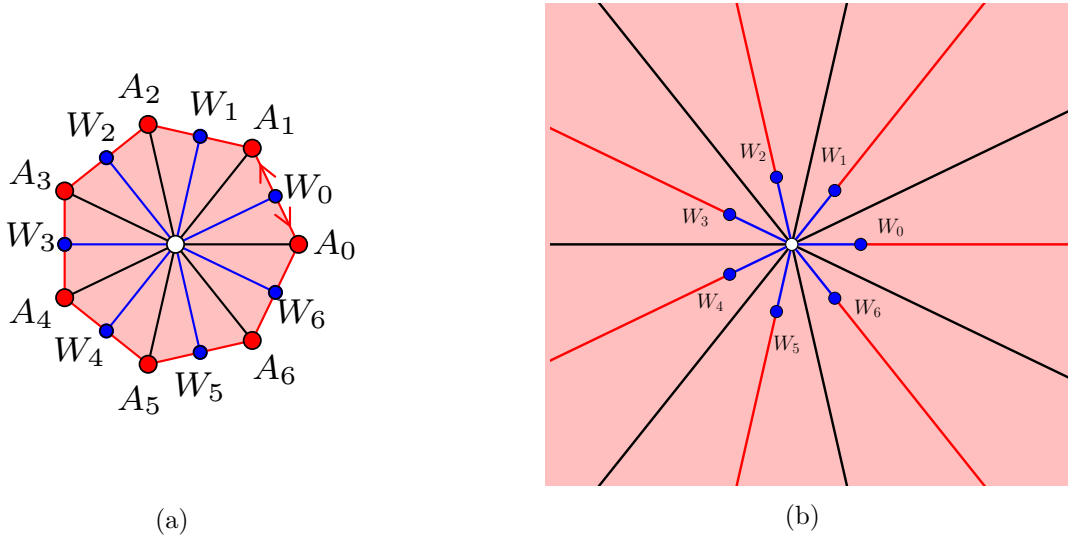

\begin{subfigure}{.5\textwidth}
  \centering
  \includestandalone[width=0.8\linewidth]{S71_v2}
  \caption{}
  \label{fig:sphere_triang2}
\end{subfigure}
\begin{subfigure}{.5\textwidth}
  \centering
  \includestandalone[width=0.85\linewidth]{sphere_triang}
  \caption{}
  \label{fig:sphere_triang1}
\end{subfigure}%
\caption{The image of the triangulation $\mathcal{T}$ under the hyperelliptic cover on $\P^1$ (a) in flat coordinate, (b) in complex coordinate.}
\label{fig:sphere_triang}
\end{figure}

\begin{proof}
We begin by determining the vertices of the triangulation of the sphere in complex coordinate. Let $O \in \P^1$ denote the image of $O$ and $B \in X$, and let $A \in \P^1$ denote the image of $A$ and $C \in X$. Recall that the automorphism $f: X \to X$ has order $N$ (see Table~\ref{table1}). Moreover, $f$ has only two orbits of length less then $N$: $\{O,B\}$ and $\{A,C\}$. At the same time $f$ descends under $\pi_\eta$ to the rotation map $\rho: \P^1 \to \P^1$, given by $x \mapsto \xi_N \cdot x$, whose only two orbits of length less than $N$ are $\{0\}$ and $\{\infty\}$. Therefore $\{0,\infty\} = \{O,A\} \subset \P^1$, and we choose to put $O$ at $0$ and $A$ at $\infty$. Note that for apolar and dipolar curves this choice is due to an extra automorphism of $X$, which is explained Remark~\ref{remark}, and for polar curves this is the only correct choice, since $\pi_\eta: X\to\P^1$ has two preimages over zero. The remaining vertices $W_i$ are the images of the Weierstrass points at the intersections of diagonals of the rhombi, therefore they are the $N$-th roots of unity.

Now that we have determined the vertices of the triangulation on the sphere, it remains to determine its edges. For this we need to take a further quotient and follow what happens with the triangulation in flat metric.

The flat metric on $(X,\omega_k)$ descends under $\pi_\eta$ to the flat metric on the sphere given by replacing each rhombus $OA_iB_iC_i$ with a triangle $OA_iC_i$ and folding each side $A_iC_i$ in half (see Figure~\ref{fig:sphere_triang2}). We omit some identifications on the picture: each $W_iA_i$ is identified with $W_iA_{i+1}$ by $\pi$-rotation around $W_i$. It follows that the the resulting flat metric is given by a meromorphic quadratic differential $q_k$ on $\P^1$.

Furthermore, the quotient of $(\P^1, q_k)$ under the rotation map $\rho$ is an orbifold with a flat metric obtained by identifying sides of two identical right angle triangles (see Figure~\ref{fig:biliard_triang2}). The angles can be calculated from the rhombus angle (see Table~\ref{table2}), they are given in (\ref{angles1}), (\ref{angles2}) and (\ref{angles3}). In  complex coordinate the vertices $O, W$ and $A$ of the triangles are $0, 1$ and $\infty$ respectively.

\begin{figure}[H]
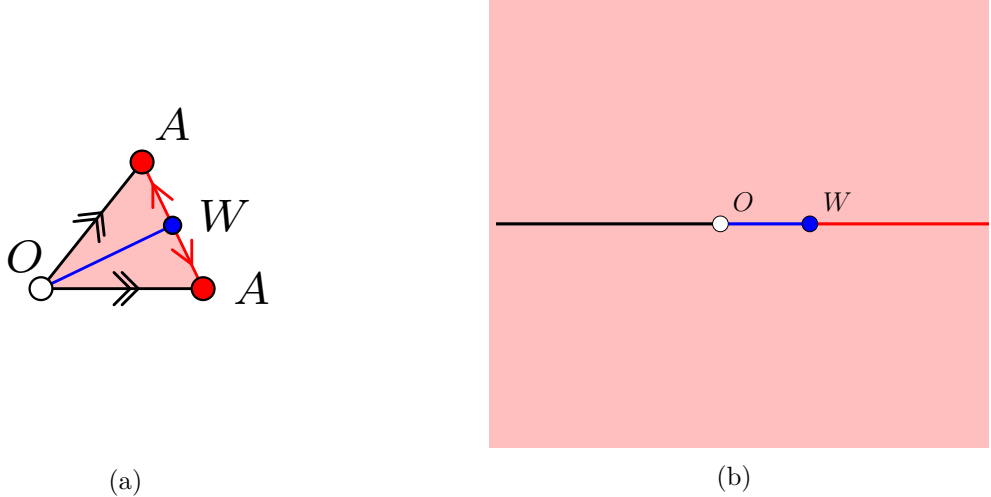

\begin{subfigure}{.48\textwidth}
  \centering
  \includestandalone[width=0.8\linewidth]{S71_v3}
  \caption{}
  \label{fig:biliard_triang2}
\end{subfigure}
\begin{subfigure}{.48\textwidth}
  \centering
  \includestandalone[width=0.85\linewidth]{biliard_triang}
  \caption{}
  \label{fig:biliard_triang1}
\end{subfigure}%
\caption{The image of the triangulation $\mathcal{T}$ under the hyperelliptic projection on $\P^1$ (a)  in flat coordinate, (b) in complex coordinate.}
\label{fig:biliard_triang}
\end{figure}

Now note that there is an anti-holomorphic map given by reflection in $OW$ line which swaps two triangles and fixes the vertices and the edges the triangulation. Such anti-holomorphic map on $\P^1$ is unique up to automorphism, and is given by $z \mapsto \overline z$. Therefore the edges of the triangulation are real line intervals: $AO = (-\infty, 0), OW = (0,1)$ and $WA = (1,\infty)$ (see Figure~\ref{fig:biliard_triang1}). Lifting these edges via the rotation quotient covering map one obtains the triangulation of the sphere as on Figure~\ref{fig:sphere_triang1}. This proves the first part of the proposition, and the second part is its immediate corollary.

Finally, for the third part it remains to notice that the unfolding procedure for the discussed right angle triangle follows the pattern of the composed covers from above. For background on unfolding see \cite{Mtab}.
\end{proof}

\bibliographystyle{amsplain} 
	\bibliography{biblio.bib}

\Addresses
\end{document}

%% file: tikz/apolar.tex
\definecolor{qqqqff}{rgb}{0.,0.,1.}
\begin{tikzpicture}[line cap=round,line join=round,>=triangle 45,x=1.0cm,y=1.0cm, scale=0.8]
\clip(-0.1413043739048595,-8.559354226792589) rectangle (17.832649848271679,6.626411464684457);
\begin{scriptsize}
\draw [line width=2.pt] (7.,-1.) circle (6.cm);
\draw [shift={(7.,12.)},line width=1.2pt]  plot[domain=4.279981204814152:5.144796755955228,variable=\t]({1.*14.317821063276352*cos(\t r)+0.*14.317821063276352*sin(\t r)},{0.*14.317821063276352*cos(\t r)+1.*14.317821063276352*sin(\t r)});
\draw [shift={(7.,-15.)},line width=1.2pt,dash pattern=on 3pt off 3pt]  plot[domain=1.1659045405098132:1.9756881130799802,variable=\t]({1.*15.231546211727817*cos(\t r)+0.*15.231546211727817*sin(\t r)},{0.*15.231546211727817*cos(\t r)+1.*15.231546211727817*sin(\t r)});
\draw [fill=qqqqff] (13.,-1.) circle (5pt);
\draw [fill=qqqqff] (1.,-1.) circle (5pt);
\draw [fill=qqqqff] (4.4228139650481255,-2.0839664917681837) circle (5pt);
\draw [fill=qqqqff] (9.526589927063482,-2.0931310694416414) circle (5pt);
\draw [fill=qqqqff] (4.4886018940427475,0.023078231620653256) circle (5pt);
\draw [fill=qqqqff] (9.469957072946125,0.029947174152125555) circle (5pt);
\draw [fill=black] (7.,5.) circle (3.0pt);
\draw [fill=black] (7.,-7.) circle (3.0pt);
\end{scriptsize}
\draw (6.3,6.3) node[anchor=north west] {\Huge $\infty$};
\draw (6.7,-7.403344381766654) node[anchor=north west] {\Huge 0};
\draw (13.5,-0.5) node[anchor=north west] {\Huge $2g+2$};
\end{tikzpicture}

%% file: tikz/polar.tex
\definecolor{qqqqff}{rgb}{0.,0.,1.}
\begin{tikzpicture}[line cap=round,line join=round,>=triangle 45,x=1.0cm,y=1.0cm, scale=0.8]
\clip(-0.1413043739048595,-8.559354226792589) rectangle (17.832649848271679,6.626411464684457);
\begin{scriptsize}
\draw [line width=2.pt] (7.,-1.) circle (6.cm);
\draw [shift={(7.,12.)},line width=1.2pt]  plot[domain=4.279981204814152:5.144796755955228,variable=\t]({1.*14.317821063276352*cos(\t r)+0.*14.317821063276352*sin(\t r)},{0.*14.317821063276352*cos(\t r)+1.*14.317821063276352*sin(\t r)});
\draw [shift={(7.,-15.)},line width=1.2pt,dash pattern=on 3pt off 3pt]  plot[domain=1.1659045405098132:1.9756881130799802,variable=\t]({1.*15.231546211727817*cos(\t r)+0.*15.231546211727817*sin(\t r)},{0.*15.231546211727817*cos(\t r)+1.*15.231546211727817*sin(\t r)});
\draw [fill=qqqqff] (2.2435576968584128,-1.5046753614029953) circle (5pt);
\draw [fill=qqqqff] (11.795564339449529,-1.4908325416261832) circle (5pt);
\draw [fill=qqqqff] (4.547976827331866,0.03288336815989901) circle (5pt);
\draw [fill=qqqqff] (9.469957072946125,0.029947174152125555) circle (5pt);
\draw [fill=qqqqff] (7.,5.) circle (5pt);
\draw [fill=black] (7.,-7.) circle (3pt);
\draw [fill=qqqqff] (7.004924163344435,-2.317820216520925) circle (5pt);
\end{scriptsize}
\draw (6.261577831043,6.2) node[anchor=north west] {\Huge$ \infty$};
\draw (6.73,-7.32344381766654) node[anchor=north west] {\Huge 0};
\draw (13.35272786872678,-0.31704181853) node[anchor=north west] {\Huge $2g+1$};
\end{tikzpicture}

%% file: tikz/dipolar.tex
\definecolor{qqqqff}{rgb}{0.,0.,1.}
\begin{tikzpicture}[line cap=round,line join=round,>=triangle 45,x=1.cm,y=1.cm, scale=0.8]
\begin{scriptsize}
\clip(-0.1413043739048595,-8.559354226792589) rectangle (17.832649848271679,6.626411464684457);
\draw [line width=2.pt] (7.,-1.) circle (6.cm);
\draw [shift={(7.,12.)},line width=1.2pt]  plot[domain=4.279981204814152:5.144796755955228,variable=\t]({1.*14.317821063276352*cos(\t r)+0.*14.317821063276352*sin(\t r)},{0.*14.317821063276352*cos(\t r)+1.*14.317821063276352*sin(\t r)});
\draw [shift={(7.,-15.)},line width=1.2pt,dash pattern=on 3pt off 3pt]  plot[domain=1.1659045405098132:1.9756881130799802,variable=\t]({1.*15.231546211727817*cos(\t r)+0.*15.231546211727817*sin(\t r)},{0.*15.231546211727817*cos(\t r)+1.*15.231546211727817*sin(\t r)});
\draw [fill=qqqqff] (4.,-2.) circle (5pt);
\draw [fill=qqqqff] (10.,-2.) circle (5pt);
\draw [fill=qqqqff] (4.012848024777291,-0.06424012388646183) circle (5pt);
\draw [fill=qqqqff] (9.987151975222707,-0.06424012388646183) circle (5pt);
\draw [fill=qqqqff] (7.,5.) circle (5pt);
\draw [fill=qqqqff] (7.,-7.) circle (5pt);
\end{scriptsize}
\draw (6.35,6.1) node[anchor=north west] {\Huge$\infty$};
\draw (6.64204109,-7.403344381766654) node[anchor=north west] {\Huge0};
\draw (13.6272786872678,-0.3704181853) node[anchor=north west] {\Huge$2g$};
\end{tikzpicture}

%% file: tikz/S71.tex
\definecolor{ffffff}{rgb}{1.,1.,1.}
\definecolor{ffqqqq}{rgb}{1.,0.,0.}
\begin{tikzpicture}[line cap=round,line join=round,>=triangle 45,x=1.cm,y=1.cm]
\clip(4.5,-3.3) rectangle (9.7,1.5);
\fill[line width=0.8pt,color=ffqqqq,fill=ffqqqq,fill opacity=0.25] (7.,-1.) -- (8.,-1.) -- (8.623489801858733,-0.21816851753197058) -- (7.623489801858733,-0.2181685175319702) -- cycle;
\fill[line width=0.8pt,color=ffqqqq,fill=ffqqqq,fill opacity=0.25] (7.,-1.) -- (7.623489801858733,-0.2181685175319702) -- (7.400968867902419,0.756759394649853) -- (6.777479066043686,-0.02507208781817649) -- cycle;
\fill[line width=0.8pt,color=ffqqqq,fill=ffqqqq,fill opacity=0.25] (7.,-1.) -- (6.777479066043686,-0.02507208781817649) -- (5.876510198141267,0.40881165129938135) -- (6.099031132097581,-0.5661162608824418) -- cycle;
\fill[line width=0.8pt,color=ffqqqq,fill=ffqqqq,fill opacity=0.25] (7.,-1.) -- (6.099031132097581,-0.5661162608824418) -- (5.198062264195162,-1.) -- (6.099031132097581,-1.433883739117558) -- cycle;
\fill[line width=0.8pt,color=ffqqqq,fill=ffqqqq,fill opacity=0.25] (7.,-1.) -- (6.099031132097581,-1.433883739117558) -- (5.876510198141267,-2.4088116512993816) -- (6.777479066043686,-1.9749279121818235) -- cycle;
\fill[line width=0.8pt,color=ffqqqq,fill=ffqqqq,fill opacity=0.25] (7.,-1.) -- (6.777479066043686,-1.9749279121818235) -- (7.400968867902419,-2.7567593946498534) -- (7.623489801858733,-1.78183148246803) -- cycle;
\fill[line width=0.8pt,color=ffqqqq,fill=ffqqqq,fill opacity=0.25] (7.,-1.) -- (7.623489801858733,-1.78183148246803) -- (8.623489801858733,-1.7818314824680301) -- (8.,-1.) -- cycle;
\draw [line width=0.8pt] (7.,-1.)-- (8.,-1.);
\draw [line width=0.8pt] (7.,-1.)-- (7.623489801858733,-0.2181685175319702);
\draw [line width=0.8pt] (8.623489801858733,-0.21816851753197058)-- (7.623489801858733,-0.2181685175319702);
\draw [line width=0.8pt] (8.623489801858733,-0.21816851753197058)-- (8.,-1.);
\draw [line width=0.8pt] (7.,-1.)-- (8.,-1.);
\draw [line width=0.8pt] (8.,-1.)-- (8.623489801858733,-0.21816851753197058);
\draw [line width=0.8pt] (8.623489801858733,-0.21816851753197058)-- (7.623489801858733,-0.2181685175319702);
\draw [line width=0.8pt] (7.623489801858733,-0.2181685175319702)-- (7.,-1.);
\draw [line width=0.8pt] (7.,-1.)-- (7.623489801858733,-0.2181685175319702);
\draw [line width=0.8pt] (7.623489801858733,-0.2181685175319702)-- (7.400968867902419,0.756759394649853);
\draw [line width=0.8pt] (7.400968867902419,0.756759394649853)-- (6.777479066043686,-0.02507208781817649);
\draw [line width=0.8pt] (6.777479066043686,-0.02507208781817649)-- (7.,-1.);
\draw [line width=0.8pt] (7.,-1.)-- (6.777479066043686,-0.02507208781817649);
\draw [line width=0.8pt] (6.777479066043686,-0.02507208781817649)-- (5.876510198141267,0.40881165129938135);
\draw [line width=0.8pt] (5.876510198141267,0.40881165129938135)-- (6.099031132097581,-0.5661162608824418);
\draw [line width=0.8pt] (6.099031132097581,-0.5661162608824418)-- (7.,-1.);
\draw [line width=0.8pt] (7.,-1.)-- (6.099031132097581,-0.5661162608824418);
\draw [line width=0.8pt] (6.099031132097581,-0.5661162608824418)-- (5.198062264195162,-1.);
\draw [line width=0.8pt] (5.198062264195162,-1.)-- (6.099031132097581,-1.433883739117558);
\draw [line width=0.8pt] (6.099031132097581,-1.433883739117558)-- (7.,-1.);
\draw [line width=0.8pt] (7.,-1.)-- (6.099031132097581,-1.433883739117558);
\draw [line width=0.8pt] (6.099031132097581,-1.433883739117558)-- (5.876510198141267,-2.4088116512993816);
\draw [line width=0.8pt] (5.876510198141267,-2.4088116512993816)-- (6.777479066043686,-1.9749279121818235);
\draw [line width=0.8pt] (6.777479066043686,-1.9749279121818235)-- (7.,-1.);
\draw [line width=0.8pt] (7.,-1.)-- (6.777479066043686,-1.9749279121818235);
\draw [line width=0.8pt] (6.777479066043686,-1.9749279121818235)-- (7.400968867902419,-2.7567593946498534);
\draw [line width=0.8pt] (7.400968867902419,-2.7567593946498534)-- (7.623489801858733,-1.78183148246803);
\draw [line width=0.8pt] (7.623489801858733,-1.78183148246803)-- (7.,-1.);
\draw [line width=0.8pt] (7.,-1.)-- (7.623489801858733,-1.78183148246803);
\draw [line width=0.8pt] (7.623489801858733,-1.78183148246803)-- (8.623489801858733,-1.7818314824680301);
\draw [line width=0.8pt] (8.623489801858733,-1.7818314824680301)-- (8.,-1.);
\draw [line width=0.8pt] (8.,-1.)-- (7.,-1.);

\def\ra{15}

\draw (7.3,-0.85) node {\scriptsize $O$};

\draw ({7+0.9*cos(-\ra)},{-1+0.9*sin(-\ra)}) node {\scriptsize $C_6$};
\draw ({7+0.9*cos(360/7-\ra)},{-1+0.9*sin(360/7-\ra)}) node {\scriptsize $C_0$};
\draw ({7+0.9*cos(360*2/7-\ra)},{-1+0.9*sin(360*2/7-\ra)}) node {\scriptsize $C_1$};
\draw ({7+0.9*cos(360*3/7-\ra)},{-1+0.9*sin(360*3/7-\ra)}) node {\scriptsize $C_2$};
\draw ({7+0.9*cos(360*4/7-\ra)},{-1+0.9*sin(360*4/7-\ra)}) node {\scriptsize $C_3$};
\draw ({7+0.9*cos(360*5/7-\ra)},{-1+0.9*sin(360*5/7-\ra)}) node {\scriptsize $C_4$};
\draw ({7+0.9*cos(360*6/7-\ra)},{-1+0.9*sin(360*6/7-\ra)}) node {\scriptsize $C_5$};

\draw ({7+0.9*cos(\ra)},{-1+0.9*sin(\ra)}) node {\scriptsize $A_0$};
\draw ({7+0.9*cos(360/7+\ra)},{-1+0.9*sin(360/7+\ra)}) node {\scriptsize $A_1$};
\draw ({7+0.9*cos(360*2/7+\ra)},{-1+0.9*sin(360*2/7+\ra)}) node {\scriptsize $A_2$};
\draw ({7+0.9*cos(360*3/7+\ra)},{-1+0.9*sin(360*3/7+\ra)}) node {\scriptsize $A_3$};
\draw ({7+0.9*cos(360*4/7+\ra)},{-1+0.9*sin(360*4/7+\ra)}) node {\scriptsize $A_4$};
\draw ({7+0.9*cos(360*5/7+\ra)},{-1+0.9*sin(360*5/7+\ra)}) node {\scriptsize $A_5$};
\draw ({7+0.9*cos(360*6/7+\ra)},{-1+0.9*sin(360*6/7+\ra)}) node {\scriptsize $A_6$};

\draw ({7+2.1*cos(360/14)},{-1+2.1*sin(360/14)}) node {\scriptsize $B_0$};
\draw ({7+2.1*cos(360*3/14)},{-1+2.1*sin(360*3/14)}) node {\scriptsize $B_1$};
\draw ({7+2.1*cos(360*5/14)},{-1+2.1*sin(360*5/14)}) node {\scriptsize $B_2$};
\draw ({7+2.1*cos(360*7/14)},{-1+2.1*sin(360*7/14)}) node {\scriptsize $B_3$};
\draw ({7+2.1*cos(360*9/14)},{-1+2.1*sin(360*9/14)}) node {\scriptsize $B_4$};
\draw ({7+2.1*cos(360*11/14)},{-1+2.1*sin(360*11/14)}) node {\scriptsize $B_5$};
\draw ({7+2.1*cos(360*13/14)},{-1+2.1*sin(360*13/14)}) node {\scriptsize $B_6$};

\begin{scriptsize}
\draw [fill=black] (8.623489801858733,-0.21816851753197058) circle (2pt);
\draw [fill=black] (8.623489801858733,-0.21816851753197058) circle (2pt);
\draw [fill=ffqqqq] (7.623489801858733,-0.2181685175319702) circle (2pt);
\draw [fill=black] (7.400968867902419,0.756759394649853) circle (2pt);
\draw [fill=ffqqqq] (6.777479066043686,-0.02507208781817649) circle (2pt);
\draw [fill=black] (5.876510198141267,0.40881165129938135) circle (2pt);
\draw [fill=ffqqqq] (6.099031132097581,-0.5661162608824418) circle (2pt);
\draw [fill=black] (5.198062264195162,-1.) circle (2pt);
\draw [fill=ffqqqq] (6.099031132097581,-1.433883739117558) circle (2pt);
\draw [fill=black] (5.876510198141267,-2.4088116512993816) circle (2pt);
\draw [fill=ffqqqq] (6.777479066043686,-1.9749279121818235) circle (2pt);
\draw [fill=black] (7.400968867902419,-2.7567593946498534) circle (2pt);
\draw [fill=ffffff] (7.,-1.) circle (2pt);
\draw [fill=ffqqqq] (7.623489801858733,-1.78183148246803) circle (2pt);
\draw [fill=black] (8.623489801858733,-1.7818314824680301) circle (2pt);
\draw [fill=ffqqqq] (8.,-1.) circle (2pt);
\end{scriptsize}
\end{tikzpicture}

%% file: tikz/S72.tex
\definecolor{ffffff}{rgb}{1.,1.,1.}
\definecolor{ffqqqq}{rgb}{1.,0.,0.}
\definecolor{qqqqff}{rgb}{0.,0.,1.}
\begin{tikzpicture}[line cap=round,line join=round,>=triangle 45,x=1.2cm,y=1.2cm]
\clip(5.4,-3.75) rectangle (9,-0.3);
\fill[line width=0.8pt,color=qqqqff,fill=qqqqff,fill opacity=0.25] (7.,-2.) -- (6.777479066043686,-1.0250720878181765) -- (7.777479066043686,-1.0250720878181765) -- (8.,-2.) -- cycle;
\fill[line width=0.8pt,color=qqqqff,fill=qqqqff,fill opacity=0.25] (7.,-2.) -- (6.099031132097581,-2.433883739117558) -- (5.876510198141267,-1.4589558269357343) -- (6.777479066043686,-1.0250720878181765) -- cycle;
\fill[line width=0.8pt,color=qqqqff,fill=qqqqff,fill opacity=0.25] (7.,-2.) -- (7.623489801858733,-2.7818314824680295) -- (6.722520933956314,-3.2157152215855875) -- (6.099031132097581,-2.433883739117558) -- cycle;
\fill[line width=0.8pt,color=qqqqff,fill=qqqqff,fill opacity=0.25] (7.,-2.) -- (7.623489801858733,-1.2181685175319705) -- (8.246979603717467,-2.) -- (7.623489801858733,-2.7818314824680295) -- cycle;
\fill[line width=0.8pt,color=qqqqff,fill=qqqqff,fill opacity=0.25] (7.,-2.) -- (6.099031132097581,-1.566116260882442) -- (6.722520933956315,-0.7842847784144125) -- (7.623489801858733,-1.2181685175319705) -- cycle;
\fill[line width=0.8pt,color=qqqqff,fill=qqqqff,fill opacity=0.25] (7.,-2.) -- (6.777479066043686,-2.9749279121818235) -- (5.876510198141267,-2.5410441730642646) -- (6.099031132097581,-1.566116260882442) -- cycle;
\fill[line width=0.8pt,color=qqqqff,fill=qqqqff,fill opacity=0.25] (7.,-2.) -- (8.,-2.) -- (7.777479066043685,-2.9749279121818235) -- (6.777479066043686,-2.9749279121818235) -- cycle;
\draw [line width=0.8pt] (7.,-2.)-- (8.,-2.);
\draw [line width=0.8pt] (7.,-2.)-- (6.777479066043686,-1.0250720878181765);
\draw [line width=0.8pt] (7.777479066043686,-1.0250720878181765)-- (6.777479066043686,-1.0250720878181765);
\draw [line width=0.8pt] (7.777479066043686,-1.0250720878181765)-- (8.,-2.);
\draw [line width=0.8pt] (7.,-2.)-- (6.777479066043686,-1.0250720878181765);
\draw [line width=0.8pt] (6.777479066043686,-1.0250720878181765)-- (7.777479066043686,-1.0250720878181765);
\draw [line width=0.8pt] (7.777479066043686,-1.0250720878181765)-- (8.,-2.);
\draw [line width=0.8pt] (8.,-2.)-- (7.,-2.);
\draw [line width=0.8pt] (7.,-2.)-- (6.099031132097581,-2.433883739117558);
\draw [line width=0.8pt] (6.099031132097581,-2.433883739117558)-- (5.876510198141267,-1.4589558269357343);
\draw [line width=0.8pt] (5.876510198141267,-1.4589558269357343)-- (6.777479066043686,-1.0250720878181765);
\draw [line width=0.8pt] (6.777479066043686,-1.0250720878181765)-- (7.,-2.);
\draw [line width=0.8pt] (7.,-2.)-- (7.623489801858733,-2.7818314824680295);
\draw [line width=0.8pt] (7.623489801858733,-2.7818314824680295)-- (6.722520933956314,-3.2157152215855875);
\draw [line width=0.8pt] (6.722520933956314,-3.2157152215855875)-- (6.099031132097581,-2.433883739117558);
\draw [line width=0.8pt] (6.099031132097581,-2.433883739117558)-- (7.,-2.);
\draw [line width=0.8pt] (7.,-2.)-- (7.623489801858733,-1.2181685175319705);
\draw [line width=0.8pt] (7.623489801858733,-1.2181685175319705)-- (8.246979603717467,-2.);
\draw [line width=0.8pt] (8.246979603717467,-2.)-- (7.623489801858733,-2.7818314824680295);
\draw [line width=0.8pt] (7.623489801858733,-2.7818314824680295)-- (7.,-2.);
\draw [line width=0.8pt] (7.,-2.)-- (6.099031132097581,-1.566116260882442);
\draw [line width=0.8pt] (6.099031132097581,-1.566116260882442)-- (6.722520933956315,-0.7842847784144125);
\draw [line width=0.8pt] (6.722520933956315,-0.7842847784144125)-- (7.623489801858733,-1.2181685175319705);
\draw [line width=0.8pt] (7.623489801858733,-1.2181685175319705)-- (7.,-2.);
\draw [line width=0.8pt] (7.,-2.)-- (6.777479066043686,-2.9749279121818235);
\draw [line width=0.8pt] (6.777479066043686,-2.9749279121818235)-- (5.876510198141267,-2.5410441730642646);
\draw [line width=0.8pt] (5.876510198141267,-2.5410441730642646)-- (6.099031132097581,-1.566116260882442);
\draw [line width=0.8pt] (6.099031132097581,-1.566116260882442)-- (7.,-2.);
\draw [line width=0.8pt] (7.,-2.)-- (8.,-2.);
\draw [line width=0.8pt] (8.,-2.)-- (7.777479066043685,-2.9749279121818235);
\draw [line width=0.8pt] (7.777479066043685,-2.9749279121818235)-- (6.777479066043686,-2.9749279121818235);
\draw [line width=0.8pt] (6.777479066043686,-2.9749279121818235)-- (7.,-2.);

\def\ra{14.5}
\def\ca{0.79}
\def\da{1.55}

\draw (7.3,-1.85) node {\scriptsize $O$};

\draw ({7+\ca*cos(-\ra)},{-2+\ca*sin(-\ra)}) node {\scriptsize $C_6$};
\draw ({7+\ca*cos(720/7-\ra)},{-2+\ca*sin(720/7-\ra)}) node {\scriptsize $C_0$};
\draw ({7+\ca*cos(720*2/7-\ra)},{-2+\ca*sin(720*2/7-\ra)}) node {\scriptsize $C_1$};
\draw ({7+\ca*cos(720*3/7-\ra)},{-2+\ca*sin(720*3/7-\ra)}) node {\scriptsize $C_2$};
\draw ({7+\ca*cos(720*4/7-\ra)},{-2+\ca*sin(720*4/7-\ra)}) node {\scriptsize $C_3$};
\draw ({7+\ca*cos(720*5/7-\ra)},{-2+\ca*sin(720*5/7-\ra)}) node {\scriptsize $C_4$};
\draw ({7+\ca*cos(720*6/7-\ra)},{-2+\ca*sin(720*6/7-\ra)}) node {\scriptsize $C_5$};

\draw ({7+\ca*cos(\ra)},{-2+\ca*sin(\ra)}) node {\scriptsize $A_0$};
\draw ({7+\ca*cos(720/7+\ra)},{-2+\ca*sin(720/7+\ra)}) node {\scriptsize $A_1$};
\draw ({7+\ca*cos(720*2/7+\ra)},{-2+\ca*sin(720*2/7+\ra)}) node {\scriptsize $A_2$};
\draw ({7+\ca*cos(720*3/7+\ra)},{-2+\ca*sin(720*3/7+\ra)}) node {\scriptsize $A_3$};
\draw ({7+\ca*cos(720*4/7+\ra)},{-2+\ca*sin(720*4/7+\ra)}) node {\scriptsize $A_4$};
\draw ({7+\ca*cos(720*5/7+\ra)},{-2+\ca*sin(720*5/7+\ra)}) node {\scriptsize $A_5$};
\draw ({7+\ca*cos(720*6/7+\ra)},{-2+\ca*sin(720*6/7+\ra)}) node {\scriptsize $A_6$};

\draw ({7+\da*cos(720/14)},{-2+\da*sin(720/14)}) node {\scriptsize $B_0$};
\draw ({7+\da*cos(720*3/14)},{-2+\da*sin(720*3/14)}) node {\scriptsize $B_1$};
\draw ({7+\da*cos(720*5/14)},{-2+\da*sin(720*5/14)}) node {\scriptsize $B_2$};
\draw ({7+\da*cos(720*7/14)},{-2+\da*sin(720*7/14)}) node {\scriptsize $B_3$};
\draw ({7+\da*cos(720*9/14)},{-2+\da*sin(720*9/14)}) node {\scriptsize $B_4$};
\draw ({7+\da*cos(720*11/14)},{-2+\da*sin(720*11/14)}) node {\scriptsize $B_5$};
\draw ({7+\da*cos(720*13/14)},{-2+\da*sin(720*13/14)}) node {\scriptsize $B_6$};

\begin{scriptsize}
\draw [fill=black] (7.777479066043686,-1.0250720878181765) circle (2pt);
\draw [fill=black] (5.876510198141267,-1.4589558269357343) circle (2pt);
\draw [fill=ffqqqq] (6.777479066043686,-1.0250720878181765) circle (2pt);
\draw [fill=black] (6.722520933956314,-3.2157152215855875) circle (2pt);
\draw [fill=ffqqqq] (6.099031132097581,-2.433883739117558) circle (2pt);
\draw [fill=black] (8.246979603717467,-2.) circle (2pt);
\draw [fill=ffqqqq] (7.623489801858733,-2.7818314824680295) circle (2pt);
\draw [fill=black] (6.722520933956315,-0.7842847784144125) circle (2pt);
\draw [fill=ffqqqq] (7.623489801858733,-1.2181685175319705) circle (2pt);
\draw [fill=black] (5.876510198141267,-2.5410441730642646) circle (2pt);
\draw [fill=ffqqqq] (6.099031132097581,-1.566116260882442) circle (2pt);
\draw [fill=ffffff] (7.,-2.) circle (2pt);
\draw [fill=ffqqqq] (8.,-2.) circle (2pt);
\draw [fill=black] (7.777479066043685,-2.9749279121818235) circle (2pt);
\draw [fill=ffqqqq] (6.777479066043686,-2.9749279121818235) circle (2pt);
\end{scriptsize}
\end{tikzpicture}

%% file: tikz/S73.tex
\definecolor{sqsqsq}{rgb}{0.12549019607843137,0.12549019607843137,0.12549019607843137}
\definecolor{ffqqqq}{rgb}{1.,0.,0.}
\definecolor{ffffff}{rgb}{1.,1.,1.}
\begin{tikzpicture}[line cap=round,line join=round,>=triangle 45,x=2.0cm,y=2.0cm]
\clip(5.7,-3.2) rectangle (8.5,-0.6);
\fill[line width=0.8pt,color=sqsqsq,fill=sqsqsq,fill opacity=0.25] (7.,-2.) -- (6.099031132097581,-1.5661162608824417) -- (7.099031132097581,-1.5661162608824417) -- (8.,-2.) -- cycle;
\fill[line width=0.8pt,color=sqsqsq,fill=sqsqsq,fill opacity=0.25] (7.,-2.) -- (6.777479066043686,-2.9749279121818235) -- (6.554958132087371,-2.) -- (6.777479066043686,-1.02501080878181765) -- cycle;
\fill[line width=0.8pt,color=sqsqsq,fill=sqsqsq,fill opacity=0.25] (7.,-2.) -- (8.,-2.) -- (7.099031132097581,-2.433883739117558) -- (6.099031132097581,-2.433883739117558) -- cycle;
\fill[line width=0.8pt,color=sqsqsq,fill=sqsqsq,fill opacity=0.25] (7.,-2.) -- (6.777479066043686,-1.02501080878181765) -- (7.400968867902419,-1.8069035702862062) -- (7.623489801858733,-2.7818314824680295) -- cycle;
\fill[line width=0.8pt,color=sqsqsq,fill=sqsqsq,fill opacity=0.25] (7.,-2.) -- (6.099031132097581,-2.433883739117558) -- (6.722520933956314,-1.6520522566495284) -- (7.623489801858733,-1.2181685175319705) -- cycle;
\fill[line width=0.8pt,color=sqsqsq,fill=sqsqsq,fill opacity=0.25] (7.,-2.) -- (7.623489801858733,-2.7818314824680295) -- (6.722520933956314,-2.347947743350472) -- (6.099031132097581,-1.566116260882442) -- cycle;
\fill[line width=0.8pt,color=sqsqsq,fill=sqsqsq,fill opacity=0.25] (7.,-2.) -- (7.623489801858733,-1.2181685175319705) -- (7.400968867902419,-2.193096429713794) -- (6.777479066043686,-2.9749279121818235) -- cycle;
\draw [line width=0.8pt] (7.,-2.)-- (8.,-2.);
\draw [line width=0.8pt] (7.,-2.)-- (6.099031132097581,-1.5661162608824417);
\draw [line width=0.8pt] (7.099031132097581,-1.5661162608824417)-- (6.099031132097581,-1.5661162608824417);
\draw [line width=0.8pt] (7.,-2.)-- (6.099031132097581,-1.5661162608824417);
\draw [line width=0.8pt] (6.099031132097581,-1.5661162608824417)-- (7.099031132097581,-1.5661162608824417);
\draw [line width=0.8pt] (7.099031132097581,-1.5661162608824417)-- (8.,-2.);
\draw [line width=0.8pt] (8.,-2.)-- (7.,-2.);
\draw [line width=0.8pt] (7.,-2.)-- (6.777479066043686,-2.9749279121818235);
\draw [line width=0.8pt] (6.777479066043686,-2.9749279121818235)-- (6.554958132087371,-2.);
\draw [line width=0.8pt] (6.554958132087371,-2.)-- (6.777479066043686,-1.02501080878181765);
\draw [line width=0.8pt] (6.777479066043686,-1.02501080878181765)-- (7.,-2.);
\draw [line width=0.8pt] (7.,-2.)-- (8.,-2.);
\draw [line width=0.8pt] (8.,-2.)-- (7.099031132097581,-2.433883739117558);
\draw [line width=0.8pt] (7.099031132097581,-2.433883739117558)-- (6.099031132097581,-2.433883739117558);
\draw [line width=0.8pt] (6.099031132097581,-2.433883739117558)-- (7.,-2.);
\draw [line width=0.8pt] (7.,-2.)-- (6.777479066043686,-1.02501080878181765);
\draw [line width=0.8pt] (6.777479066043686,-1.02501080878181765)-- (7.400968867902419,-1.8069035702862062);
\draw [line width=0.8pt] (7.400968867902419,-1.8069035702862062)-- (7.623489801858733,-2.7818314824680295);
\draw [line width=0.8pt] (7.623489801858733,-2.7818314824680295)-- (7.,-2.);
\draw [line width=0.8pt] (7.,-2.)-- (6.099031132097581,-2.433883739117558);
\draw [line width=0.8pt] (6.099031132097581,-2.433883739117558)-- (6.722520933956314,-1.6520522566495284);
\draw [line width=0.8pt] (6.722520933956314,-1.6520522566495284)-- (7.623489801858733,-1.2181685175319705);
\draw [line width=0.8pt] (7.623489801858733,-1.2181685175319705)-- (7.,-2.);
\draw [line width=0.8pt] (7.,-2.)-- (7.623489801858733,-2.7818314824680295);
\draw [line width=0.8pt] (7.623489801858733,-2.7818314824680295)-- (6.722520933956314,-2.347947743350472);
\draw [line width=0.8pt] (6.722520933956314,-2.347947743350472)-- (6.099031132097581,-1.566116260882442);
\draw [line width=0.8pt] (6.099031132097581,-1.566116260882442)-- (7.,-2.);
\draw [line width=0.8pt] (7.,-2.)-- (7.623489801858733,-1.2181685175319705);
\draw [line width=0.8pt] (7.623489801858733,-1.2181685175319705)-- (7.400968867902419,-2.193096429713794);
\draw [line width=0.8pt] (7.400968867902419,-2.193096429713794)-- (6.777479066043686,-2.9749279121818235);
\draw [line width=0.8pt] (6.777479066043686,-2.9749279121818235)-- (7.,-2.);

\def\ra{8}
\def\ca{1.05}
\def\da{0.32}

\draw (7.143,-1.925) node {\scriptsize $O$};

\draw ({7+\ca*cos(-\ra)},{-2+\ca*sin(-\ra)}) node {\scriptsize $C_6$};
\draw ({7+\ca*cos(1080/7-\ra)},{-2+\ca*sin(1080/7-\ra)}) node {\scriptsize $C_0$};
\draw ({7+\ca*cos(1080*2/7-\ra)},{-2+\ca*sin(1080*2/7-\ra)}) node {\scriptsize $C_1$};
\draw ({7+\ca*cos(1080*3/7-\ra)},{-2+\ca*sin(1080*3/7-\ra)}) node {\scriptsize $C_2$};
\draw ({7+\ca*cos(1080*4/7-\ra)},{-2+\ca*sin(1080*4/7-\ra)}) node {\scriptsize $C_3$};
\draw ({7+\ca*cos(1080*5/7-\ra)},{-2+\ca*sin(1080*5/7-\ra)}) node {\scriptsize $C_4$};
\draw ({7+\ca*cos(1080*6/7-\ra)},{-2+\ca*sin(1080*6/7-\ra)}) node {\scriptsize $C_5$};

\draw ({7+\ca*cos(\ra)},{-2+\ca*sin(\ra)}) node {\scriptsize $A_0$};
\draw ({7+\ca*cos(1080/7+\ra)},{-2+\ca*sin(1080/7+\ra)}) node {\scriptsize $A_1$};
\draw ({7+\ca*cos(1080*2/7+\ra)},{-2+\ca*sin(1080*2/7+\ra)}) node {\scriptsize $A_2$};
\draw ({7+\ca*cos(1080*3/7+\ra)},{-2+\ca*sin(1080*3/7+\ra)}) node {\scriptsize $A_3$};
\draw ({7+\ca*cos(1080*4/7+\ra)},{-2+\ca*sin(1080*4/7+\ra)}) node {\scriptsize $A_4$};
\draw ({7+\ca*cos(1080*5/7+\ra)},{-2+\ca*sin(1080*5/7+\ra)}) node {\scriptsize $A_5$};
\draw ({7+\ca*cos(1080*6/7+\ra)},{-2+\ca*sin(1080*6/7+\ra)}) node {\scriptsize $A_6$};

\draw ({7+\da*cos(1080/14)},{-2+\da*sin(1080/14)}) node {\scriptsize $B_0$};
\draw ({7+\da*cos(1080*3/14)},{-2+\da*sin(1080*3/14)}) node {\scriptsize $B_1$};
\draw ({7+\da*cos(1080*5/14)},{-2+\da*sin(1080*5/14)}) node {\scriptsize $B_2$};
\draw ({7+\da*cos(1080*7/14)},{-2+\da*sin(1080*7/14)}) node {\scriptsize $B_3$};
\draw ({7+\da*cos(1080*9/14)},{-2+\da*sin(1080*9/14)}) node {\scriptsize $B_4$};
\draw ({7+\da*cos(1080*11/14)},{-2+\da*sin(1080*11/14)}) node {\scriptsize $B_5$};
\draw ({7+\da*cos(1080*13/14)},{-2+\da*sin(1080*13/14)}) node {\scriptsize $B_6$};

\begin{scriptsize}
\draw [fill=ffffff] (7.,-2.) circle (1.8pt);
\draw [fill=ffqqqq] (8.,-2.) circle (1.8pt);
\draw [fill=black] (6.099031132097581,-1.5661162608824417) circle (1.8pt);
\draw [fill=black] (7.099031132097581,-1.5661162608824417) circle (1.8pt);
\draw [fill=black] (7.099031132097581,-1.5661162608824417) circle (1.8pt);
\draw [fill=black] (6.554958132087371,-2.) circle (1.8pt);
\draw [fill=ffqqqq] (6.777479066043686,-1.02501080878181765) circle (1.8pt);
\draw [fill=black] (7.099031132097581,-2.433883739117558) circle (1.8pt);
\draw [fill=ffqqqq] (6.099031132097581,-2.433883739117558) circle (1.8pt);
\draw [fill=black] (7.400968867902419,-1.8069035702862062) circle (1.8pt);
\draw [fill=ffqqqq] (7.623489801858733,-2.7818314824680295) circle (1.8pt);
\draw [fill=black] (6.722520933956314,-1.6520522566495284) circle (1.8pt);
\draw [fill=ffqqqq] (7.623489801858733,-1.2181685175319705) circle (1.8pt);
\draw [fill=black] (6.722520933956314,-2.347947743350472) circle (1.8pt);
\draw [fill=ffqqqq] (6.099031132097581,-1.566116260882442) circle (1.8pt);
\draw [fill=ffffff] (7.,-2.) circle (2pt);
\draw [fill=ffqqqq] (7.623489801858733,-1.2181685175319705) circle (1.8pt);
\draw [fill=black] (7.400968867902419,-2.193096429713794) circle (1.8pt);
\draw [fill=ffqqqq] (6.777479066043686,-2.9749279121818235) circle (1.8pt);
\end{scriptsize}
\end{tikzpicture}

%% file: tikz/H61.tex
\definecolor{ffffff}{rgb}{1.,1.,1.}
\definecolor{zzttqq}{rgb}{0.6,0.2,0.}
\begin{tikzpicture}[line cap=round,line join=round,>=triangle 45,x=0.9cm,y=0.9cm]
\clip(-9.,-2.5) rectangle (13,10);
\fill[line width=2.pt,color=black,fill=zzttqq,fill opacity=0.10000000149011612] (2.,-1.) -- (7.,-1.) -- (11.330127018922195,1.5) -- (6.330127018922194,1.5) -- cycle;
\fill[line width=2.pt,color=black,fill=zzttqq,fill opacity=0.10000000149011612] (2.,-1.) -- (6.330127018922194,1.5) -- (8.830127018922195,5.830127018922193) -- (4.5,3.3301270189221928) -- cycle;
\fill[line width=2.pt,color=black,fill=zzttqq,fill opacity=0.10000000149011612] (2.,-1.) -- (4.5,3.3301270189221928) -- (4.5,8.330127018922193) -- (2.,4.) -- cycle;
\fill[line width=2.pt,color=black,fill=zzttqq,fill opacity=0.10000000149011612] (2.,-1.) -- (2.,4.) -- (-0.5,8.330127018922193) -- (-0.5,3.3301270189221945) -- cycle;
\fill[line width=2.pt,color=black,fill=zzttqq,fill opacity=0.10000000149011612] (2.,-1.) -- (-0.5,3.3301270189221945) -- (-4.83012701892219,5.8301270189221945) -- (-2.330127018922192,1.5) -- cycle;
\fill[line width=2.pt,color=black,fill=zzttqq,fill opacity=0.10000000149011612] (2.,-1.) -- (-2.330127018922192,1.5) -- (-7.330127018922191,1.5) -- (-3.,-1.) -- cycle;
\fill[line width=2.pt,color=black,fill=zzttqq,fill opacity=0.10000000149011612] (25.,-1.) -- (31.,-1.) -- (31.,5.) -- (25.,5.) -- cycle;
\fill[line width=2.pt,color=black,fill=zzttqq,fill opacity=0.10000000149011612] (25.,-1.) -- (19.,-1.) -- (19.,5.) -- (25.,5.) -- cycle;
\fill[line width=2.pt,color=black,fill=zzttqq,fill opacity=0.10000000149011612] (19.,-1.) -- (25.,-1.) -- (25.,-7.) -- (19.,-7.) -- cycle;
\fill[line width=2.pt,color=black,fill=zzttqq,fill opacity=0.10000000149011612] (25.553418261060155,-1.5828792378820884) -- (25.553418261060155,-7.582879237882084) -- (31.553418261060166,-7.582879237882084) -- (31.554554903030382,-1.6024815837338857) -- cycle;
\fill[line width=2.pt,color=black,fill=zzttqq,fill opacity=0.10000000149011612] (25.553418261060155,-1.5828792378820884) -- (31.554554903030382,-1.6024815837338857) -- (31.553418261060166,4.417120762117914) -- (25.553418261060155,4.417120762117914) -- cycle;
\fill[line width=2.pt,color=black,fill=zzttqq,fill opacity=0.10000000149011612] (25.553418261060155,4.417120762117914) -- (19.553418261060155,4.417120762117914) -- (19.553418261060155,-1.5828792378820884) -- (25.553418261060155,-1.5828792378820884) -- cycle;
\fill[line width=2.pt,color=black,fill=zzttqq,fill opacity=0.10000000149011612] (39.80384757729337,1.) -- (45.,-2.) -- (51.,-2.) -- (45.80384757729337,1.) -- cycle;
\fill[line width=2.pt,color=black,fill=zzttqq,fill opacity=0.10000000149011612] (48.,-7.196152422706631) -- (45.,-2.) -- (39.80384757729337,1.) -- (42.80384757729337,-4.196152422706632) -- cycle;
\fill[line width=2.pt,color=black,fill=zzttqq,fill opacity=0.10000000149011612] (45.,4.) -- (45.,-2.) -- (48.,-7.196152422706631) -- (48.,-1.1961524227066316) -- cycle;
\fill[line width=2.pt,color=black,fill=zzttqq,fill opacity=0.10000000149011612] (42.,-7.196152422706632) -- (45.,-2.) -- (45.,4.) -- (42.,-1.1961524227066325) -- cycle;
\fill[line width=2.pt,color=black,fill=zzttqq,fill opacity=0.10000000149011612] (50.19615242270663,1.) -- (45.,-2.) -- (42.,-7.196152422706632) -- (47.19615242270663,-4.196152422706631) -- cycle;
\fill[line width=2.pt,color=black,fill=zzttqq,fill opacity=0.10000000149011612] (39.,-2.) -- (45.,-2.) -- (50.19615242270663,1.) -- (44.19615242270663,1.) -- cycle;
\draw [line width=2.pt,color=black] (2.,-1.)-- (7.,-1.);
\draw [line width=2.pt,color=black] (7.,-1.)-- (11.330127018922195,1.5);
\draw [line width=2.pt,color=black] (11.330127018922195,1.5)-- (6.330127018922194,1.5);
\draw [line width=2.pt,color=black] (6.330127018922194,1.5)-- (2.,-1.);
\draw [line width=2.pt,color=black] (2.,-1.)-- (6.330127018922194,1.5);
\draw [line width=2.pt,color=black] (6.330127018922194,1.5)-- (8.830127018922195,5.830127018922193);
\draw [line width=2.pt,color=black] (8.830127018922195,5.830127018922193)-- (4.5,3.3301270189221928);
\draw [line width=2.pt,color=black] (4.5,3.3301270189221928)-- (2.,-1.);
\draw [line width=2.pt,color=black] (2.,-1.)-- (4.5,3.3301270189221928);
\draw [line width=2.pt,color=black] (4.5,3.3301270189221928)-- (4.5,8.330127018922193);
\draw [line width=2.pt,color=black] (4.5,8.330127018922193)-- (2.,4.);
\draw [line width=2.pt,color=black] (2.,4.)-- (2.,-1.);
\draw [line width=2.pt,color=black] (2.,-1.)-- (2.,4.);
\draw [line width=2.pt,color=black] (2.,4.)-- (-0.5,8.330127018922193);
\draw [line width=2.pt,color=black] (-0.5,8.330127018922193)-- (-0.5,3.3301270189221945);
\draw [line width=2.pt,color=black] (-0.5,3.3301270189221945)-- (2.,-1.);
\draw [line width=2.pt,color=black] (2.,-1.)-- (-0.5,3.3301270189221945);
\draw [line width=2.pt,color=black] (-0.5,3.3301270189221945)-- (-4.83012701892219,5.8301270189221945);
\draw [line width=2.pt,color=black] (-4.83012701892219,5.8301270189221945)-- (-2.330127018922192,1.5);
\draw [line width=2.pt,color=black] (-2.330127018922192,1.5)-- (2.,-1.);
\draw [line width=2.pt,color=black] (2.,-1.)-- (-2.330127018922192,1.5);
\draw [line width=2.pt,color=black] (-2.330127018922192,1.5)-- (-7.330127018922191,1.5);
\draw [line width=2.pt,color=black] (-7.330127018922191,1.5)-- (-3.,-1.);
\draw [line width=2.pt,color=black] (-3.,-1.)-- (2.,-1.);
\draw [line width=2.pt,color=black] (25.,-1.)-- (31.,-1.);
\draw [line width=2.pt,color=black] (31.,-1.)-- (31.,5.);
\draw [line width=2.pt,color=black] (31.,5.)-- (25.,5.);
\draw [line width=2.pt,color=black] (25.,5.)-- (25.,-1.);
\draw [line width=2.pt,color=black] (25.,-1.)-- (19.,-1.);
\draw [line width=2.pt,color=black] (19.,-1.)-- (19.,5.);
\draw [line width=2.pt,color=black] (19.,5.)-- (25.,5.);
\draw [line width=2.pt,color=black] (25.,5.)-- (25.,-1.);
\draw [line width=2.pt,color=black] (19.,-1.)-- (25.,-1.);
\draw [line width=2.pt,color=black] (25.,-1.)-- (25.,-7.);
\draw [line width=2.pt,color=black] (25.,-7.)-- (19.,-7.);
\draw [line width=2.pt,color=black] (19.,-7.)-- (19.,-1.);
\draw [line width=2.pt,color=black] (25.553418261060155,-1.5828792378820884)-- (25.553418261060155,-7.582879237882084);
\draw [line width=2.pt,color=black] (25.553418261060155,-7.582879237882084)-- (31.553418261060166,-7.582879237882084);
\draw [line width=2.pt,color=black] (31.553418261060166,-7.582879237882084)-- (31.554554903030382,-1.6024815837338857);
\draw [line width=2.pt,color=black] (31.554554903030382,-1.6024815837338857)-- (25.553418261060155,-1.5828792378820884);
\draw [line width=2.pt,color=black] (25.553418261060155,-1.5828792378820884)-- (31.554554903030382,-1.6024815837338857);
\draw [line width=2.pt,color=black] (31.554554903030382,-1.6024815837338857)-- (31.553418261060166,4.417120762117914);
\draw [line width=2.pt,color=black] (31.553418261060166,4.417120762117914)-- (25.553418261060155,4.417120762117914);
\draw [line width=2.pt,color=black] (25.553418261060155,4.417120762117914)-- (25.553418261060155,-1.5828792378820884);
\draw [line width=2.pt,color=black] (25.553418261060155,4.417120762117914)-- (19.553418261060155,4.417120762117914);
\draw [line width=2.pt,color=black] (19.553418261060155,4.417120762117914)-- (19.553418261060155,-1.5828792378820884);
\draw [line width=2.pt,color=black] (19.553418261060155,-1.5828792378820884)-- (25.553418261060155,-1.5828792378820884);
\draw [line width=2.pt,color=black] (25.553418261060155,-1.5828792378820884)-- (25.553418261060155,4.417120762117914);
\draw [line width=2.pt,color=black] (39.80384757729337,1.)-- (45.,-2.);
\draw [line width=2.pt,color=black] (45.,-2.)-- (51.,-2.);
\draw [line width=2.pt,color=black] (51.,-2.)-- (45.80384757729337,1.);
\draw [line width=2.pt,color=black] (45.80384757729337,1.)-- (39.80384757729337,1.);
\draw [line width=2.pt,color=black] (48.,-7.196152422706631)-- (45.,-2.);
\draw [line width=2.pt,color=black] (45.,-2.)-- (39.80384757729337,1.);
\draw [line width=2.pt,color=black] (39.80384757729337,1.)-- (42.80384757729337,-4.196152422706632);
\draw [line width=2.pt,color=black] (42.80384757729337,-4.196152422706632)-- (48.,-7.196152422706631);
\draw [line width=2.pt,color=black] (45.,4.)-- (45.,-2.);
\draw [line width=2.pt,color=black] (45.,-2.)-- (48.,-7.196152422706631);
\draw [line width=2.pt,color=black] (48.,-7.196152422706631)-- (48.,-1.1961524227066316);
\draw [line width=2.pt,color=black] (48.,-1.1961524227066316)-- (45.,4.);
\draw [line width=2.pt,color=black] (42.,-7.196152422706632)-- (45.,-2.);
\draw [line width=2.pt,color=black] (45.,-2.)-- (45.,4.);
\draw [line width=2.pt,color=black] (45.,4.)-- (42.,-1.1961524227066325);
\draw [line width=2.pt,color=black] (42.,-1.1961524227066325)-- (42.,-7.196152422706632);
\draw [line width=2.pt,color=black] (50.19615242270663,1.)-- (45.,-2.);
\draw [line width=2.pt,color=black] (45.,-2.)-- (42.,-7.196152422706632);
\draw [line width=2.pt,color=black] (42.,-7.196152422706632)-- (47.19615242270663,-4.196152422706631);
\draw [line width=2.pt,color=black] (47.19615242270663,-4.196152422706631)-- (50.19615242270663,1.);
\draw [line width=2.pt,color=black] (39.,-2.)-- (45.,-2.);
\draw [line width=2.pt,color=black] (45.,-2.)-- (50.19615242270663,1.);
\draw [line width=2.pt,color=black] (50.19615242270663,1.)-- (44.19615242270663,1.);
\draw [line width=2.pt,color=black] (44.19615242270663,1.)-- (39.,-2.);
\begin{scriptsize}
\draw [fill=ffffff] (2.,-1.) circle (6.0pt);
\draw [fill=black] (7.,-1.) circle (6.0pt);
\draw [fill=black] (6.330127018922194,1.5) circle (6.0pt);
\draw [fill=ffffff] (11.330127018922195,1.5) circle (6.0pt);
\draw [fill=ffffff] (8.830127018922195,5.830127018922193) circle (6.0pt);
\draw [fill=black] (4.5,3.3301270189221928) circle (6.0pt);
\draw [fill=ffffff] (4.5,8.330127018922193) circle (6.0pt);
\draw [fill=black] (2.,4.) circle (6.0pt);
\draw [fill=ffffff] (-0.5,8.330127018922193) circle (6.0pt);
\draw [fill=black] (-0.5,3.3301270189221945) circle (6.0pt);
\draw [fill=ffffff] (-4.83012701892219,5.8301270189221945) circle (6.0pt);
\draw [fill=black] (-2.330127018922192,1.5) circle (6.0pt);
\draw [fill=ffffff] (-7.330127018922191,1.5) circle (6.0pt);
\draw [fill=black] (-3.,-1.) circle (6.0pt);
\draw [fill=ffffff] (25.,-1.) circle (6.0pt);
\draw [fill=black] (31.,-1.) circle (6.0pt);
\draw [fill=ffffff] (31.,5.) circle (6.0pt);
\draw [fill=black] (25.,5.) circle (6.0pt);
\draw [fill=black] (19.,-1.) circle (6.0pt);
\draw [fill=ffffff] (19.,5.) circle (6.0pt);
\draw [fill=black] (25.,-7.) circle (6.0pt);
\draw [fill=ffffff] (19.,-7.) circle (6.0pt);
\draw [fill=ffffff] (25.553418261060155,-1.5828792378820884) circle (6.0pt);
\draw [fill=black] (25.553418261060155,-7.582879237882084) circle (6.0pt);
\draw [fill=ffffff] (31.553418261060166,-7.582879237882084) circle (6.0pt);
\draw [fill=black] (31.554554903030382,-1.6024815837338857) circle (6.0pt);
\draw [fill=ffffff] (31.553418261060166,4.417120762117914) circle (6.0pt);
\draw [fill=black] (25.553418261060155,4.417120762117914) circle (6.0pt);
\draw [fill=ffffff] (19.553418261060155,4.417120762117914) circle (6.0pt);
\draw [fill=black] (19.553418261060155,-1.5828792378820884) circle (6.0pt);
\draw [fill=black] (51.,-2.) circle (6.0pt);
\draw [fill=black] (39.80384757729337,1.) circle (6.0pt);
\draw [fill=black] (45.80384757729337,1.) circle (6.0pt);
\draw [fill=black] (48.,-7.196152422706631) circle (6.0pt);
\draw [fill=black] (42.80384757729337,-4.196152422706632) circle (6.0pt);
\draw [fill=black] (45.,4.) circle (6.0pt);
\draw [fill=black] (48.,-1.1961524227066316) circle (6.0pt);
\draw [fill=black] (42.,-7.196152422706632) circle (6.0pt);
\draw [fill=black] (42.,-1.1961524227066325) circle (6.0pt);
\draw [fill=black] (50.19615242270663,1.) circle (6.0pt);
\draw [fill=black] (47.19615242270663,-4.196152422706631) circle (6.0pt);
\draw [fill=black] (39.,-2.) circle (6.0pt);
\draw [fill=black] (45.,-2.) circle (6.0pt);
\draw [fill=black] (44.19615242270663,1.) circle (6.0pt);
\end{scriptsize}

\def\ra{8.5}

\draw (2,-1.85) node {\huge $O$};

\draw ({2+5.1*cos(\ra)},{-1+5.1*sin(\ra)}) node {\huge $A_0$};
\draw ({2+5.1*cos(30+\ra)},{-1+5.1*sin(30+\ra)}) node {\huge $A_1$};
\draw ({2+5.1*cos(2*30+\ra)},{-1+5.1*sin(2*30+\ra)}) node {\huge $A_2$};
\draw ({2+5.1*cos(3*30+\ra)},{-1+5.1*sin(3*30+\ra)}) node {\huge $A_3$};
\draw ({2+5.1*cos(4*30+\ra)},{-1+5.1*sin(4*30+\ra)}) node {\huge $A_4$};
\draw ({2+5.1*cos(5*30+\ra)},{-1+5.1*sin(5*30+\ra)}) node {\huge $A_5$};
\draw ({2+5.1*cos(6*30+\ra)},{-1+5.1*sin(6*30+\ra)}) node {\huge $A_6$};

\draw ({2+5.1*cos(30-\ra)},{-1+5.1*sin(30-\ra)}) node {\huge $C_0$};
\draw ({2+5.1*cos(2*30-\ra)},{-1+5.1*sin(2*30-\ra)}) node {\huge $C_1$};
\draw ({2+5.1*cos(3*30-\ra)},{-1+5.1*sin(3*30-\ra)}) node {\huge $C_2$};
\draw ({2+5.1*cos(4*30-\ra)},{-1+5.1*sin(4*30-\ra)}) node {\huge $C_3$};
\draw ({2+5.1*cos(5*30-\ra)},{-1+5.1*sin(5*30-\ra)}) node {\huge $C_4$};
\draw ({2+5.1*cos(6*30-\ra)},{-1+5.1*sin(6*30-\ra)}) node {\huge $C_5$};

\draw ({2+10.6*cos(15)},{-1+10.6*sin(15)}) node {\huge $C_0$};
\draw ({2+10.6*cos(30+15)},{-1+10.6*sin(30+15)}) node {\huge $C_1$};
\draw ({2+10.6*cos(2*30+15)},{-1+10.6*sin(2*30+15)}) node {\huge $C_2$};
\draw ({2+10.6*cos(3*30+15)},{-1+10.6*sin(3*30+15)}) node {\huge $C_3$};
\draw ({2+10.6*cos(4*30+15)},{-1+10.6*sin(4*30+15)}) node {\huge $C_4$};
\draw ({2+10.6*cos(5*30+15)},{-1+10.6*sin(5*30+15)}) node {\huge $C_5$};

\end{tikzpicture}

%% file: tikz/H62.tex
\definecolor{ffffff}{rgb}{1.,1.,1.}
\definecolor{zzttqq}{rgb}{0.6,0.2,0.}
\begin{tikzpicture}[line cap=round,line join=round,>=triangle 45,x=0.8cm,y=0.8cm]
\clip(13.,-8.) rectangle (37.5,6.);
\fill[line width=2.pt,color=black,fill=zzttqq,fill opacity=0.10000000149011612] (2.,-1.) -- (7.,-1.) -- (11.330127018922195,1.5) -- (6.330127018922194,1.5) -- cycle;
\fill[line width=2.pt,color=black,fill=zzttqq,fill opacity=0.10000000149011612] (2.,-1.) -- (6.330127018922194,1.5) -- (8.830127018922195,5.830127018922193) -- (4.5,3.3301270189221928) -- cycle;
\fill[line width=2.pt,color=black,fill=zzttqq,fill opacity=0.10000000149011612] (2.,-1.) -- (4.5,3.3301270189221928) -- (4.5,8.330127018922193) -- (2.,4.) -- cycle;
\fill[line width=2.pt,color=black,fill=zzttqq,fill opacity=0.10000000149011612] (2.,-1.) -- (2.,4.) -- (-0.5,8.330127018922193) -- (-0.5,3.3301270189221945) -- cycle;
\fill[line width=2.pt,color=black,fill=zzttqq,fill opacity=0.10000000149011612] (2.,-1.) -- (-0.5,3.3301270189221945) -- (-4.83012701892219,5.8301270189221945) -- (-2.330127018922192,1.5) -- cycle;
\fill[line width=2.pt,color=black,fill=zzttqq,fill opacity=0.10000000149011612] (2.,-1.) -- (-2.330127018922192,1.5) -- (-7.330127018922191,1.5) -- (-3.,-1.) -- cycle;
\fill[line width=2.pt,color=black,fill=zzttqq,fill opacity=0.10000000149011612] (25.,-1.) -- (31.,-1.) -- (31.,5.) -- (25.,5.) -- cycle;
\fill[line width=2.pt,color=black,fill=zzttqq,fill opacity=0.10000000149011612] (25.,-1.) -- (19.,-1.) -- (19.,5.) -- (25.,5.) -- cycle;
\fill[line width=2.pt,color=black,fill=zzttqq,fill opacity=0.10000000149011612] (19.,-1.) -- (25.,-1.) -- (25.,-7.) -- (19.,-7.) -- cycle;
\fill[line width=2.pt,color=black,fill=zzttqq,fill opacity=0.10000000149011612] (25.553418261060155,-1.5828792378820884) -- (25.553418261060155,-7.582879237882084) -- (31.553418261060166,-7.582879237882084) -- (31.554554903030382,-1.6024815837338857) -- cycle;
\fill[line width=2.pt,color=black,fill=zzttqq,fill opacity=0.10000000149011612] (25.553418261060155,-1.5828792378820884) -- (31.554554903030382,-1.6024815837338857) -- (31.553418261060166,4.417120762117914) -- (25.553418261060155,4.417120762117914) -- cycle;
\fill[line width=2.pt,color=black,fill=zzttqq,fill opacity=0.10000000149011612] (25.553418261060155,4.417120762117914) -- (19.553418261060155,4.417120762117914) -- (19.553418261060155,-1.5828792378820884) -- (25.553418261060155,-1.5828792378820884) -- cycle;
\fill[line width=2.pt,color=black,fill=zzttqq,fill opacity=0.10000000149011612] (39.80384757729337,1.) -- (45.,-2.) -- (51.,-2.) -- (45.80384757729337,1.) -- cycle;
\fill[line width=2.pt,color=black,fill=zzttqq,fill opacity=0.10000000149011612] (48.,-7.196152422706631) -- (45.,-2.) -- (39.80384757729337,1.) -- (42.80384757729337,-4.196152422706632) -- cycle;
\fill[line width=2.pt,color=black,fill=zzttqq,fill opacity=0.10000000149011612] (45.,4.) -- (45.,-2.) -- (48.,-7.196152422706631) -- (48.,-1.1961524227066316) -- cycle;
\fill[line width=2.pt,color=black,fill=zzttqq,fill opacity=0.10000000149011612] (42.,-7.196152422706632) -- (45.,-2.) -- (45.,4.) -- (42.,-1.1961524227066325) -- cycle;
\fill[line width=2.pt,color=black,fill=zzttqq,fill opacity=0.10000000149011612] (50.19615242270663,1.) -- (45.,-2.) -- (42.,-7.196152422706632) -- (47.19615242270663,-4.196152422706631) -- cycle;
\fill[line width=2.pt,color=black,fill=zzttqq,fill opacity=0.10000000149011612] (39.,-2.) -- (45.,-2.) -- (50.19615242270663,1.) -- (44.19615242270663,1.) -- cycle;
\draw [line width=2.pt,color=black] (2.,-1.)-- (7.,-1.);
\draw [line width=2.pt,color=black] (7.,-1.)-- (11.330127018922195,1.5);
\draw [line width=2.pt,color=black] (11.330127018922195,1.5)-- (6.330127018922194,1.5);
\draw [line width=2.pt,color=black] (6.330127018922194,1.5)-- (2.,-1.);
\draw [line width=2.pt,color=black] (2.,-1.)-- (6.330127018922194,1.5);
\draw [line width=2.pt,color=black] (6.330127018922194,1.5)-- (8.830127018922195,5.830127018922193);
\draw [line width=2.pt,color=black] (8.830127018922195,5.830127018922193)-- (4.5,3.3301270189221928);
\draw [line width=2.pt,color=black] (4.5,3.3301270189221928)-- (2.,-1.);
\draw [line width=2.pt,color=black] (2.,-1.)-- (4.5,3.3301270189221928);
\draw [line width=2.pt,color=black] (4.5,3.3301270189221928)-- (4.5,8.330127018922193);
\draw [line width=2.pt,color=black] (4.5,8.330127018922193)-- (2.,4.);
\draw [line width=2.pt,color=black] (2.,4.)-- (2.,-1.);
\draw [line width=2.pt,color=black] (2.,-1.)-- (2.,4.);
\draw [line width=2.pt,color=black] (2.,4.)-- (-0.5,8.330127018922193);
\draw [line width=2.pt,color=black] (-0.5,8.330127018922193)-- (-0.5,3.3301270189221945);
\draw [line width=2.pt,color=black] (-0.5,3.3301270189221945)-- (2.,-1.);
\draw [line width=2.pt,color=black] (2.,-1.)-- (-0.5,3.3301270189221945);
\draw [line width=2.pt,color=black] (-0.5,3.3301270189221945)-- (-4.83012701892219,5.8301270189221945);
\draw [line width=2.pt,color=black] (-4.83012701892219,5.8301270189221945)-- (-2.330127018922192,1.5);
\draw [line width=2.pt,color=black] (-2.330127018922192,1.5)-- (2.,-1.);
\draw [line width=2.pt,color=black] (2.,-1.)-- (-2.330127018922192,1.5);
\draw [line width=2.pt,color=black] (-2.330127018922192,1.5)-- (-7.330127018922191,1.5);
\draw [line width=2.pt,color=black] (-7.330127018922191,1.5)-- (-3.,-1.);
\draw [line width=2.pt,color=black] (-3.,-1.)-- (2.,-1.);
\draw [line width=2.pt,color=black] (25.,-1.)-- (31.,-1.);
\draw [line width=2.pt,color=black] (31.,-1.)-- (31.,5.);
\draw [line width=2.pt,color=black] (31.,5.)-- (25.,5.);
\draw [line width=2.pt,color=black] (25.,5.)-- (25.,-1.);
\draw [line width=2.pt,color=black] (25.,-1.)-- (19.,-1.);
\draw [line width=2.pt,color=black] (19.,-1.)-- (19.,5.);
\draw [line width=2.pt,color=black] (19.,5.)-- (25.,5.);
\draw [line width=2.pt,color=black] (25.,5.)-- (25.,-1.);
\draw [line width=2.pt,color=black] (19.,-1.)-- (25.,-1.);
\draw [line width=2.pt,color=black] (25.,-1.)-- (25.,-7.);
\draw [line width=2.pt,color=black] (25.,-7.)-- (19.,-7.);
\draw [line width=2.pt,color=black] (19.,-7.)-- (19.,-1.);
\draw [line width=2.pt,color=black] (25.553418261060155,-1.5828792378820884)-- (25.553418261060155,-7.582879237882084);
\draw [line width=2.pt,color=black] (25.553418261060155,-7.582879237882084)-- (31.553418261060166,-7.582879237882084);
\draw [line width=2.pt,color=black] (31.553418261060166,-7.582879237882084)-- (31.554554903030382,-1.6024815837338857);
\draw [line width=2.pt,color=black] (31.554554903030382,-1.6024815837338857)-- (25.553418261060155,-1.5828792378820884);
\draw [line width=2.pt,color=black] (25.553418261060155,-1.5828792378820884)-- (31.554554903030382,-1.6024815837338857);
\draw [line width=2.pt,color=black] (31.554554903030382,-1.6024815837338857)-- (31.553418261060166,4.417120762117914);
\draw [line width=2.pt,color=black] (31.553418261060166,4.417120762117914)-- (25.553418261060155,4.417120762117914);
\draw [line width=2.pt,color=black] (25.553418261060155,4.417120762117914)-- (25.553418261060155,-1.5828792378820884);
\draw [line width=2.pt,color=black] (25.553418261060155,4.417120762117914)-- (19.553418261060155,4.417120762117914);
\draw [line width=2.pt,color=black] (19.553418261060155,4.417120762117914)-- (19.553418261060155,-1.5828792378820884);
\draw [line width=2.pt,color=black] (19.553418261060155,-1.5828792378820884)-- (25.553418261060155,-1.5828792378820884);
\draw [line width=2.pt,color=black] (25.553418261060155,-1.5828792378820884)-- (25.553418261060155,4.417120762117914);
\draw [line width=2.pt,color=black] (39.80384757729337,1.)-- (45.,-2.);
\draw [line width=2.pt,color=black] (45.,-2.)-- (51.,-2.);
\draw [line width=2.pt,color=black] (51.,-2.)-- (45.80384757729337,1.);
\draw [line width=2.pt,color=black] (45.80384757729337,1.)-- (39.80384757729337,1.);
\draw [line width=2.pt,color=black] (48.,-7.196152422706631)-- (45.,-2.);
\draw [line width=2.pt,color=black] (45.,-2.)-- (39.80384757729337,1.);
\draw [line width=2.pt,color=black] (39.80384757729337,1.)-- (42.80384757729337,-4.196152422706632);
\draw [line width=2.pt,color=black] (42.80384757729337,-4.196152422706632)-- (48.,-7.196152422706631);
\draw [line width=2.pt,color=black] (45.,4.)-- (45.,-2.);
\draw [line width=2.pt,color=black] (45.,-2.)-- (48.,-7.196152422706631);
\draw [line width=2.pt,color=black] (48.,-7.196152422706631)-- (48.,-1.1961524227066316);
\draw [line width=2.pt,color=black] (48.,-1.1961524227066316)-- (45.,4.);
\draw [line width=2.pt,color=black] (42.,-7.196152422706632)-- (45.,-2.);
\draw [line width=2.pt,color=black] (45.,-2.)-- (45.,4.);
\draw [line width=2.pt,color=black] (45.,4.)-- (42.,-1.1961524227066325);
\draw [line width=2.pt,color=black] (42.,-1.1961524227066325)-- (42.,-7.196152422706632);
\draw [line width=2.pt,color=black] (50.19615242270663,1.)-- (45.,-2.);
\draw [line width=2.pt,color=black] (45.,-2.)-- (42.,-7.196152422706632);
\draw [line width=2.pt,color=black] (42.,-7.196152422706632)-- (47.19615242270663,-4.196152422706631);
\draw [line width=2.pt,color=black] (47.19615242270663,-4.196152422706631)-- (50.19615242270663,1.);
\draw [line width=2.pt,color=black] (39.,-2.)-- (45.,-2.);
\draw [line width=2.pt,color=black] (45.,-2.)-- (50.19615242270663,1.);
\draw [line width=2.pt,color=black] (50.19615242270663,1.)-- (44.19615242270663,1.);
\draw [line width=2.pt,color=black] (44.19615242270663,1.)-- (39.,-2.);
\begin{scriptsize}
\draw [fill=black] (2.,-1.) circle (6.0pt);
\draw [fill=black] (7.,-1.) circle (6.0pt);
\draw [fill=black] (6.330127018922194,1.5) circle (6.0pt);
\draw [fill=black] (11.330127018922195,1.5) circle (6.0pt);
\draw [fill=black] (8.830127018922195,5.830127018922193) circle (6.0pt);
\draw [fill=black] (4.5,3.3301270189221928) circle (6.0pt);
\draw [fill=black] (4.5,8.330127018922193) circle (6.0pt);
\draw [fill=black] (2.,4.) circle (6.0pt);
\draw [fill=black] (-0.5,8.330127018922193) circle (6.0pt);
\draw [fill=black] (-0.5,3.3301270189221945) circle (6.0pt);
\draw [fill=black] (-4.83012701892219,5.8301270189221945) circle (6.0pt);
\draw [fill=black] (-2.330127018922192,1.5) circle (6.0pt);
\draw [fill=black] (-7.330127018922191,1.5) circle (6.0pt);
\draw [fill=black] (-3.,-1.) circle (6.0pt);
\draw [fill=ffffff] (25.,-1.) circle (6.0pt);
\draw [fill=black] (31.,-1.) circle (6.0pt);
\draw [fill=ffffff] (31.,5.) circle (6.0pt);
\draw [fill=black] (25.,5.) circle (6.0pt);
\draw [fill=black] (19.,-1.) circle (6.0pt);
\draw [fill=ffffff] (19.,5.) circle (6.0pt);
\draw [fill=black] (25.,-7.) circle (6.0pt);
\draw [fill=ffffff] (19.,-7.) circle (6.0pt);
\draw [fill=ffffff] (25.553418261060155,-1.5828792378820884) circle (6.0pt);
\draw [fill=black] (25.553418261060155,-7.582879237882084) circle (6.0pt);
\draw [fill=ffffff] (31.553418261060166,-7.582879237882084) circle (6.0pt);
\draw [fill=black] (31.554554903030382,-1.6024815837338857) circle (6.0pt);
\draw [fill=ffffff] (31.553418261060166,4.417120762117914) circle (6.0pt);
\draw [fill=black] (25.553418261060155,4.417120762117914) circle (6.0pt);
\draw [fill=ffffff] (19.553418261060155,4.417120762117914) circle (6.0pt);
\draw [fill=black] (19.553418261060155,-1.5828792378820884) circle (6.0pt);
\draw [fill=black] (51.,-2.) circle (6.0pt);
\draw [fill=black] (39.80384757729337,1.) circle (6.0pt);
\draw [fill=black] (45.80384757729337,1.) circle (6.0pt);
\draw [fill=black] (48.,-7.196152422706631) circle (6.0pt);
\draw [fill=black] (42.80384757729337,-4.196152422706632) circle (6.0pt);
\draw [fill=black] (45.,4.) circle (6.0pt);
\draw [fill=black] (48.,-1.1961524227066316) circle (6.0pt);
\draw [fill=black] (42.,-7.196152422706632) circle (6.0pt);
\draw [fill=black] (42.,-1.1961524227066325) circle (6.0pt);
\draw [fill=black] (50.19615242270663,1.) circle (6.0pt);
\draw [fill=black] (47.19615242270663,-4.196152422706631) circle (6.0pt);
\draw [fill=black] (39.,-2.) circle (6.0pt);
\draw [fill=black] (45.,-2.) circle (6.0pt);
\draw [fill=black] (44.19615242270663,1.) circle (6.0pt);
\end{scriptsize}






\end{tikzpicture}

%% file: tikz/H63.tex
\definecolor{ffffff}{rgb}{1.,1.,1.}
\definecolor{zzttqq}{rgb}{0.6,0.2,0.}
\begin{tikzpicture}[line cap=round,line join=round,>=triangle 45,x=0.9cm,y=0.9cm]
\clip(34.,-8.) rectangle (56.,5.);
\fill[line width=2.pt,color=black,fill=zzttqq,fill opacity=0.10000000149011612] (2.,-1.) -- (7.,-1.) -- (11.330127018922195,1.5) -- (6.330127018922194,1.5) -- cycle;
\fill[line width=2.pt,color=black,fill=zzttqq,fill opacity=0.10000000149011612] (2.,-1.) -- (6.330127018922194,1.5) -- (8.830127018922195,5.830127018922193) -- (4.5,3.3301270189221928) -- cycle;
\fill[line width=2.pt,color=black,fill=zzttqq,fill opacity=0.10000000149011612] (2.,-1.) -- (4.5,3.3301270189221928) -- (4.5,8.330127018922193) -- (2.,4.) -- cycle;
\fill[line width=2.pt,color=black,fill=zzttqq,fill opacity=0.10000000149011612] (2.,-1.) -- (2.,4.) -- (-0.5,8.330127018922193) -- (-0.5,3.3301270189221945) -- cycle;
\fill[line width=2.pt,color=black,fill=zzttqq,fill opacity=0.10000000149011612] (2.,-1.) -- (-0.5,3.3301270189221945) -- (-4.83012701892219,5.8301270189221945) -- (-2.330127018922192,1.5) -- cycle;
\fill[line width=2.pt,color=black,fill=zzttqq,fill opacity=0.10000000149011612] (2.,-1.) -- (-2.330127018922192,1.5) -- (-7.330127018922191,1.5) -- (-3.,-1.) -- cycle;
\fill[line width=2.pt,color=black,fill=zzttqq,fill opacity=0.10000000149011612] (25.,-1.) -- (31.,-1.) -- (31.,5.) -- (25.,5.) -- cycle;
\fill[line width=2.pt,color=black,fill=zzttqq,fill opacity=0.10000000149011612] (25.,-1.) -- (19.,-1.) -- (19.,5.) -- (25.,5.) -- cycle;
\fill[line width=2.pt,color=black,fill=zzttqq,fill opacity=0.10000000149011612] (19.,-1.) -- (25.,-1.) -- (25.,-7.) -- (19.,-7.) -- cycle;
\fill[line width=2.pt,color=black,fill=zzttqq,fill opacity=0.10000000149011612] (25.553418261060155,-1.5828792378820884) -- (25.553418261060155,-7.582879237882084) -- (31.553418261060166,-7.582879237882084) -- (31.554554903030382,-1.6024815837338857) -- cycle;
\fill[line width=2.pt,color=black,fill=zzttqq,fill opacity=0.10000000149011612] (25.553418261060155,-1.5828792378820884) -- (31.554554903030382,-1.6024815837338857) -- (31.553418261060166,4.417120762117914) -- (25.553418261060155,4.417120762117914) -- cycle;
\fill[line width=2.pt,color=black,fill=zzttqq,fill opacity=0.10000000149011612] (25.553418261060155,4.417120762117914) -- (19.553418261060155,4.417120762117914) -- (19.553418261060155,-1.5828792378820884) -- (25.553418261060155,-1.5828792378820884) -- cycle;
\fill[line width=2.pt,color=black,fill=zzttqq,fill opacity=0.10000000149011612] (39.80384757729337,1.) -- (45.,-2.) -- (51.,-2.) -- (45.80384757729337,1.) -- cycle;
\fill[line width=2.pt,color=black,fill=zzttqq,fill opacity=0.10000000149011612] (48.,-7.196152422706631) -- (45.,-2.) -- (39.80384757729337,1.) -- (42.80384757729337,-4.196152422706632) -- cycle;
\fill[line width=2.pt,color=black,fill=zzttqq,fill opacity=0.10000000149011612] (45.,4.) -- (45.,-2.) -- (48.,-7.196152422706631) -- (48.,-1.1961524227066316) -- cycle;
\fill[line width=2.pt,color=black,fill=zzttqq,fill opacity=0.10000000149011612] (42.,-7.196152422706632) -- (45.,-2.) -- (45.,4.) -- (42.,-1.1961524227066325) -- cycle;
\fill[line width=2.pt,color=black,fill=zzttqq,fill opacity=0.10000000149011612] (50.19615242270663,1.) -- (45.,-2.) -- (42.,-7.196152422706632) -- (47.19615242270663,-4.196152422706631) -- cycle;
\fill[line width=2.pt,color=black,fill=zzttqq,fill opacity=0.10000000149011612] (39.,-2.) -- (45.,-2.) -- (50.19615242270663,1.) -- (44.19615242270663,1.) -- cycle;
\draw [line width=2.pt,color=black] (2.,-1.)-- (7.,-1.);
\draw [line width=2.pt,color=black] (7.,-1.)-- (11.330127018922195,1.5);
\draw [line width=2.pt,color=black] (11.330127018922195,1.5)-- (6.330127018922194,1.5);
\draw [line width=2.pt,color=black] (6.330127018922194,1.5)-- (2.,-1.);
\draw [line width=2.pt,color=black] (2.,-1.)-- (6.330127018922194,1.5);
\draw [line width=2.pt,color=black] (6.330127018922194,1.5)-- (8.830127018922195,5.830127018922193);
\draw [line width=2.pt,color=black] (8.830127018922195,5.830127018922193)-- (4.5,3.3301270189221928);
\draw [line width=2.pt,color=black] (4.5,3.3301270189221928)-- (2.,-1.);
\draw [line width=2.pt,color=black] (2.,-1.)-- (4.5,3.3301270189221928);
\draw [line width=2.pt,color=black] (4.5,3.3301270189221928)-- (4.5,8.330127018922193);
\draw [line width=2.pt,color=black] (4.5,8.330127018922193)-- (2.,4.);
\draw [line width=2.pt,color=black] (2.,4.)-- (2.,-1.);
\draw [line width=2.pt,color=black] (2.,-1.)-- (2.,4.);
\draw [line width=2.pt,color=black] (2.,4.)-- (-0.5,8.330127018922193);
\draw [line width=2.pt,color=black] (-0.5,8.330127018922193)-- (-0.5,3.3301270189221945);
\draw [line width=2.pt,color=black] (-0.5,3.3301270189221945)-- (2.,-1.);
\draw [line width=2.pt,color=black] (2.,-1.)-- (-0.5,3.3301270189221945);
\draw [line width=2.pt,color=black] (-0.5,3.3301270189221945)-- (-4.83012701892219,5.8301270189221945);
\draw [line width=2.pt,color=black] (-4.83012701892219,5.8301270189221945)-- (-2.330127018922192,1.5);
\draw [line width=2.pt,color=black] (-2.330127018922192,1.5)-- (2.,-1.);
\draw [line width=2.pt,color=black] (2.,-1.)-- (-2.330127018922192,1.5);
\draw [line width=2.pt,color=black] (-2.330127018922192,1.5)-- (-7.330127018922191,1.5);
\draw [line width=2.pt,color=black] (-7.330127018922191,1.5)-- (-3.,-1.);
\draw [line width=2.pt,color=black] (-3.,-1.)-- (2.,-1.);
\draw [line width=2.pt,color=black] (25.,-1.)-- (31.,-1.);
\draw [line width=2.pt,color=black] (31.,-1.)-- (31.,5.);
\draw [line width=2.pt,color=black] (31.,5.)-- (25.,5.);
\draw [line width=2.pt,color=black] (25.,5.)-- (25.,-1.);
\draw [line width=2.pt,color=black] (25.,-1.)-- (19.,-1.);
\draw [line width=2.pt,color=black] (19.,-1.)-- (19.,5.);
\draw [line width=2.pt,color=black] (19.,5.)-- (25.,5.);
\draw [line width=2.pt,color=black] (25.,5.)-- (25.,-1.);
\draw [line width=2.pt,color=black] (19.,-1.)-- (25.,-1.);
\draw [line width=2.pt,color=black] (25.,-1.)-- (25.,-7.);
\draw [line width=2.pt,color=black] (25.,-7.)-- (19.,-7.);
\draw [line width=2.pt,color=black] (19.,-7.)-- (19.,-1.);
\draw [line width=2.pt,color=black] (25.553418261060155,-1.5828792378820884)-- (25.553418261060155,-7.582879237882084);
\draw [line width=2.pt,color=black] (25.553418261060155,-7.582879237882084)-- (31.553418261060166,-7.582879237882084);
\draw [line width=2.pt,color=black] (31.553418261060166,-7.582879237882084)-- (31.554554903030382,-1.6024815837338857);
\draw [line width=2.pt,color=black] (31.554554903030382,-1.6024815837338857)-- (25.553418261060155,-1.5828792378820884);
\draw [line width=2.pt,color=black] (25.553418261060155,-1.5828792378820884)-- (31.554554903030382,-1.6024815837338857);
\draw [line width=2.pt,color=black] (31.554554903030382,-1.6024815837338857)-- (31.553418261060166,4.417120762117914);
\draw [line width=2.pt,color=black] (31.553418261060166,4.417120762117914)-- (25.553418261060155,4.417120762117914);
\draw [line width=2.pt,color=black] (25.553418261060155,4.417120762117914)-- (25.553418261060155,-1.5828792378820884);
\draw [line width=2.pt,color=black] (25.553418261060155,4.417120762117914)-- (19.553418261060155,4.417120762117914);
\draw [line width=2.pt,color=black] (19.553418261060155,4.417120762117914)-- (19.553418261060155,-1.5828792378820884);
\draw [line width=2.pt,color=black] (19.553418261060155,-1.5828792378820884)-- (25.553418261060155,-1.5828792378820884);
\draw [line width=2.pt,color=black] (25.553418261060155,-1.5828792378820884)-- (25.553418261060155,4.417120762117914);
\draw [line width=2.pt,color=black] (39.80384757729337,1.)-- (45.,-2.);
\draw [line width=2.pt,color=black] (45.,-2.)-- (51.,-2.);
\draw [line width=2.pt,color=black] (51.,-2.)-- (45.80384757729337,1.);
\draw [line width=2.pt,color=black] (45.80384757729337,1.)-- (39.80384757729337,1.);
\draw [line width=2.pt,color=black] (48.,-7.196152422706631)-- (45.,-2.);
\draw [line width=2.pt,color=black] (45.,-2.)-- (39.80384757729337,1.);
\draw [line width=2.pt,color=black] (39.80384757729337,1.)-- (42.80384757729337,-4.196152422706632);
\draw [line width=2.pt,color=black] (42.80384757729337,-4.196152422706632)-- (48.,-7.196152422706631);
\draw [line width=2.pt,color=black] (45.,4.)-- (45.,-2.);
\draw [line width=2.pt,color=black] (45.,-2.)-- (48.,-7.196152422706631);
\draw [line width=2.pt,color=black] (48.,-7.196152422706631)-- (48.,-1.1961524227066316);
\draw [line width=2.pt,color=black] (48.,-1.1961524227066316)-- (45.,4.);
\draw [line width=2.pt,color=black] (42.,-7.196152422706632)-- (45.,-2.);
\draw [line width=2.pt,color=black] (45.,-2.)-- (45.,4.);
\draw [line width=2.pt,color=black] (45.,4.)-- (42.,-1.1961524227066325);
\draw [line width=2.pt,color=black] (42.,-1.1961524227066325)-- (42.,-7.196152422706632);
\draw [line width=2.pt,color=black] (50.19615242270663,1.)-- (45.,-2.);
\draw [line width=2.pt,color=black] (45.,-2.)-- (42.,-7.196152422706632);
\draw [line width=2.pt,color=black] (42.,-7.196152422706632)-- (47.19615242270663,-4.196152422706631);
\draw [line width=2.pt,color=black] (47.19615242270663,-4.196152422706631)-- (50.19615242270663,1.);
\draw [line width=2.pt,color=black] (39.,-2.)-- (45.,-2.);
\draw [line width=2.pt,color=black] (45.,-2.)-- (50.19615242270663,1.);
\draw [line width=2.pt,color=black] (50.19615242270663,1.)-- (44.19615242270663,1.);
\draw [line width=2.pt,color=black] (44.19615242270663,1.)-- (39.,-2.);
\begin{scriptsize}
\draw [fill=black] (2.,-1.) circle (6.0pt);
\draw [fill=black] (7.,-1.) circle (6.0pt);
\draw [fill=black] (6.330127018922194,1.5) circle (6.0pt);
\draw [fill=black] (11.330127018922195,1.5) circle (6.0pt);
\draw [fill=black] (8.830127018922195,5.830127018922193) circle (6.0pt);
\draw [fill=black] (4.5,3.3301270189221928) circle (6.0pt);
\draw [fill=black] (4.5,8.330127018922193) circle (6.0pt);
\draw [fill=black] (2.,4.) circle (6.0pt);
\draw [fill=black] (-0.5,8.330127018922193) circle (6.0pt);
\draw [fill=black] (-0.5,3.3301270189221945) circle (6.0pt);
\draw [fill=ffffff] (-4.83012701892219,5.8301270189221945) circle (6.0pt);
\draw [fill=black] (-2.330127018922192,1.5) circle (6.0pt);
\draw [fill=black] (-7.330127018922191,1.5) circle (6.0pt);
\draw [fill=black] (-3.,-1.) circle (6.0pt);
\draw [fill=ffffff] (25.,-1.) circle (6.0pt);
\draw [fill=black] (31.,-1.) circle (6.0pt);
\draw [fill=ffffff] (31.,5.) circle (6.0pt);
\draw [fill=black] (25.,5.) circle (6.0pt);
\draw [fill=black] (19.,-1.) circle (6.0pt);
\draw [fill=ffffff] (19.,5.) circle (6.0pt);
\draw [fill=black] (25.,-7.) circle (6.0pt);
\draw [fill=ffffff] (19.,-7.) circle (6.0pt);
\draw [fill=ffffff] (25.553418261060155,-1.5828792378820884) circle (6.0pt);
\draw [fill=black] (25.553418261060155,-7.582879237882084) circle (6.0pt);
\draw [fill=ffffff] (31.553418261060166,-7.582879237882084) circle (6.0pt);
\draw [fill=black] (31.554554903030382,-1.6024815837338857) circle (6.0pt);
\draw [fill=ffffff] (31.553418261060166,4.417120762117914) circle (6.0pt);
\draw [fill=black] (25.553418261060155,4.417120762117914) circle (6.0pt);
\draw [fill=ffffff] (19.553418261060155,4.417120762117914) circle (6.0pt);
\draw [fill=black] (19.553418261060155,-1.5828792378820884) circle (6.0pt);
\draw [fill=ffffff] (51.,-2.) circle (6.0pt);
\draw [fill=ffffff] (39.80384757729337,1.) circle (6.0pt);
\draw [fill=black] (45.80384757729337,1.) circle (6.0pt);
\draw [fill=ffffff] (48.,-7.196152422706631) circle (6.0pt);
\draw [fill=black] (42.80384757729337,-4.196152422706632) circle (6.0pt);
\draw [fill=ffffff] (45.,4.) circle (6.0pt);
\draw [fill=black] (48.,-1.1961524227066316) circle (6.0pt);
\draw [fill=ffffff] (42.,-7.196152422706632) circle (6.0pt);
\draw [fill=black] (42.,-1.1961524227066325) circle (6.0pt);
\draw [fill=ffffff] (50.19615242270663,1.) circle (6.0pt);
\draw [fill=black] (47.19615242270663,-4.196152422706631) circle (6.0pt);
\draw [fill=ffffff] (39.,-2.) circle (6.0pt);
\draw [fill=black] (45.,-2.) circle (6.0pt);
\draw [fill=black] (44.19615242270663,1.) circle (6.0pt);
\end{scriptsize}
\end{tikzpicture}

%% file: tikz/S71_v1.tex
\definecolor{ffffff}{rgb}{1.,1.,1.}
\definecolor{ffqqqq}{rgb}{1.,0.,0.}
\begin{tikzpicture}[line cap=round,line join=round,>=triangle 45,x=1.0cm,y=1.0cm]
\clip(4.5,-3.5) rectangle (9.7,1.5);
\fill[line width=0.8pt,color=ffqqqq,fill=ffqqqq,fill opacity=0.25] (7.,-1.) -- (8.,-1.) -- (8.623489801858733,-0.21816851753197058) -- (7.623489801858733,-0.2181685175319702) -- cycle;
\fill[line width=0.8pt,color=ffqqqq,fill=ffqqqq,fill opacity=0.25] (7.,-1.) -- (7.623489801858733,-0.2181685175319702) -- (7.400968867902419,0.756759394649853) -- (6.777479066043686,-0.02507208781817649) -- cycle;
\fill[line width=0.8pt,color=ffqqqq,fill=ffqqqq,fill opacity=0.25] (7.,-1.) -- (6.777479066043686,-0.02507208781817649) -- (5.876510198141267,0.40881165129938135) -- (6.099031132097581,-0.5661162608824418) -- cycle;
\fill[line width=0.8pt,color=ffqqqq,fill=ffqqqq,fill opacity=0.25] (7.,-1.) -- (6.099031132097581,-0.5661162608824418) -- (5.198062264195162,-1.) -- (6.099031132097581,-1.433883739117558) -- cycle;
\fill[line width=0.8pt,color=ffqqqq,fill=ffqqqq,fill opacity=0.25] (7.,-1.) -- (6.099031132097581,-1.433883739117558) -- (5.876510198141267,-2.4088116512993816) -- (6.777479066043686,-1.9749279121818235) -- cycle;
\fill[line width=0.8pt,color=ffqqqq,fill=ffqqqq,fill opacity=0.25] (7.,-1.) -- (6.777479066043686,-1.9749279121818235) -- (7.400968867902419,-2.7567593946498534) -- (7.623489801858733,-1.78183148246803) -- cycle;
\fill[line width=0.8pt,color=ffqqqq,fill=ffqqqq,fill opacity=0.25] (7.,-1.) -- (7.623489801858733,-1.78183148246803) -- (8.623489801858733,-1.7818314824680301) -- (8.,-1.) -- cycle;
\draw [line width=0.8pt] (7.,-1.)-- (8.,-1.);
\draw [line width=0.8pt] (7.,-1.)-- (7.623489801858733,-0.2181685175319702);
\draw [line width=0.8pt] (8.623489801858733,-0.21816851753197058)-- (7.623489801858733,-0.2181685175319702);
\draw [line width=0.8pt] (8.,-1.)-- (8.623489801858733,-0.21816851753197058);
\draw [line width=0.8pt] (8.623489801858733,-0.21816851753197058)-- (7.623489801858733,-0.2181685175319702);
\draw [line width=0.8pt] (7.623489801858733,-0.2181685175319702)-- (7.,-1.);
\draw [line width=0.8pt] (7.,-1.)-- (7.623489801858733,-0.2181685175319702);
\draw [line width=0.8pt] (7.623489801858733,-0.2181685175319702)-- (7.400968867902419,0.756759394649853);
\draw [line width=0.8pt] (7.400968867902419,0.756759394649853)-- (6.777479066043686,-0.02507208781817649);
\draw [line width=0.8pt] (6.777479066043686,-0.02507208781817649)-- (7.,-1.);
\draw [line width=0.8pt] (6.777479066043686,-0.02507208781817649)-- (5.876510198141267,0.40881165129938135);
\draw [line width=0.8pt] (5.876510198141267,0.40881165129938135)-- (6.099031132097581,-0.5661162608824418);
\draw [line width=0.8pt] (6.099031132097581,-0.5661162608824418)-- (7.,-1.);
\draw [line width=0.8pt] (7.,-1.)-- (6.099031132097581,-0.5661162608824418);
\draw [line width=0.8pt] (6.099031132097581,-0.5661162608824418)-- (5.198062264195162,-1.);
\draw [line width=0.8pt] (5.198062264195162,-1.)-- (6.099031132097581,-1.433883739117558);
\draw [line width=0.8pt] (6.099031132097581,-1.433883739117558)-- (7.,-1.);
\draw [line width=0.8pt] (7.,-1.)-- (6.099031132097581,-1.433883739117558);
\draw [line width=0.8pt] (6.099031132097581,-1.433883739117558)-- (5.876510198141267,-2.4088116512993816);
\draw [line width=0.8pt] (5.876510198141267,-2.4088116512993816)-- (6.777479066043686,-1.9749279121818235);
\draw [line width=0.8pt] (6.777479066043686,-1.9749279121818235)-- (7.,-1.);
\draw [line width=0.8pt] (7.,-1.)-- (6.777479066043686,-1.9749279121818235);
\draw [line width=0.8pt] (6.777479066043686,-1.9749279121818235)-- (7.400968867902419,-2.7567593946498534);
\draw [line width=0.8pt] (7.400968867902419,-2.7567593946498534)-- (7.623489801858733,-1.78183148246803);
\draw [line width=0.8pt] (7.623489801858733,-1.78183148246803)-- (7.,-1.);
\draw [line width=0.8pt] (7.,-1.)-- (7.623489801858733,-1.78183148246803);
\draw [line width=0.8pt] (7.623489801858733,-1.78183148246803)-- (8.623489801858733,-1.7818314824680301);
\draw [line width=0.8pt] (8.623489801858733,-1.7818314824680301)-- (8.,-1.);
\draw [line width=0.8pt] (8.,-1.)-- (7.,-1.);

\draw [line width=0.3pt] (7.623489801858733+0.5,-1.78183148246803+0.1)-- (8.623489801858733-0.5,-1.7818314824680301-0.1);

\draw [line width=0.3pt] (7.623489801858733+0.5,-0.21816851753197058+0.1)-- (8.623489801858733-0.5,-0.21816851753197058-0.1);

\draw [line width=0.3pt] (8.31+0.07,-0.61-0.07)-- (8.31-0.07,-0.61+0.07);
\draw [line width=0.3pt] (8.31+0.07-0.05,-0.61-0.07-0.05)-- (8.31-0.07-0.05,-0.61+0.07-0.05);

\draw [line width=0.3pt] (8.31+0.07+6.777479066043686-8,-0.61-0.07-0.02507208781817649+1)-- (8.31-0.07+6.777479066043686-8,-0.61+0.07-0.02507208781817649+1);
\draw [line width=0.3pt] (8.31+0.07-0.05+6.777479066043686-8,-0.61-0.07-0.05-0.02507208781817649+1)-- (8.31-0.07-0.05+6.777479066043686-8,-0.61+0.07-0.05-0.02507208781817649+1);

(7.623489801858733-8,-0.2181685175319702+1)

\def\ra{15}

\draw (7.3,-0.85) node {\scriptsize $O$};


\draw ({7+1.3*cos(0)},{-1+1.3*sin(0)}) node {\scriptsize $A$};
\draw ({7+1.3*cos(360/7)},{-1+1.3*sin(360/7)}) node {\scriptsize $A$};
\draw ({7+1.3*cos(360*2/7)},{-1+1.3*sin(360*2/7)}) node {\scriptsize $A$};
\draw ({7+1.3*cos(360*3/7)},{-1+1.3*sin(360*3/7)}) node {\scriptsize $A$};
\draw ({7+1.3*cos(360*4/7)},{-1+1.3*sin(360*4/7)}) node {\scriptsize $A$};
\draw ({7+1.3*cos(360*5/7)},{-1+1.3*sin(360*5/7)}) node {\scriptsize $A$};
\draw ({7+1.3*cos(360*6/7)},{-1+1.3*sin(360*6/7)}) node {\scriptsize $A$};

\draw ({7+2.1*cos(360/14)},{-1+2.1*sin(360/14)}) node {\scriptsize $B$};
\draw ({7+2.1*cos(360*3/14)},{-1+2.1*sin(360*3/14)}) node {\scriptsize $B$};
\draw ({7+2.1*cos(360*5/14)},{-1+2.1*sin(360*5/14)}) node {\scriptsize $B$};
\draw ({7+2.1*cos(360*7/14)},{-1+2.1*sin(360*7/14)}) node {\scriptsize $B$};
\draw ({7+2.1*cos(360*9/14)},{-1+2.1*sin(360*9/14)}) node {\scriptsize $B$};
\draw ({7+2.1*cos(360*11/14)},{-1+2.1*sin(360*11/14)}) node {\scriptsize $B$};
\draw ({7+2.1*cos(360*13/14)},{-1+2.1*sin(360*13/14)}) node {\scriptsize $B$};

\draw [line width=0.3pt, color = red] (8.,-1.)-- (7.623489801858733,-0.2181685175319702)-- (6.777479066043686,-0.02507208781817649)-- (6.099031132097581,-0.5661162608824418)-- (6.099031132097581,-1.433883739117558)-- (6.777479066043686,-1.9749279121818235)-- (7.623489801858733,-1.78183148246803)-- (8.,-1.);
\draw [line width=0.3pt, color = blue] (7,-1)-- (8.623489801858733,-0.21816851753197058);
\draw [line width=0.3pt, color = blue] (7,-1)-- (7.400968867902419,0.756759394649853);
\draw [line width=0.3pt, color = blue] (7,-1)-- (5.876510198141267,0.40881165129938135);
\draw [line width=0.3pt, color = blue] (7,-1)-- (5.198062264195162,-1.);
\draw [line width=0.3pt, color = blue] (7,-1)-- (5.876510198141267,-2.4088116512993816);
\draw [line width=0.3pt, color = blue] (7,-1)-- (7.400968867902419,-2.7567593946498534);
\draw [line width=0.3pt, color = blue] (7,-1)-- (8.623489801858733,-1.7818314824680301);

\begin{scriptsize}
\draw [fill=black] (8.623489801858733,-0.21816851753197058) circle (2pt);
\draw [fill=ffqqqq] (7.623489801858733,-0.2181685175319702) circle (2pt);
\draw [fill=black] (7.400968867902419,0.756759394649853) circle (2pt);
\draw [fill=ffqqqq] (6.777479066043686,-0.02507208781817649) circle (2pt);
\draw [fill=black] (5.876510198141267,0.40881165129938135) circle (2pt);
\draw [fill=ffqqqq] (6.099031132097581,-0.5661162608824418) circle (2pt);
\draw [fill=black] (5.198062264195162,-1.) circle (2pt);
\draw [fill=ffqqqq] (6.099031132097581,-1.433883739117558) circle (2pt);
\draw [fill=black] (5.876510198141267,-2.4088116512993816) circle (2pt);
\draw [fill=ffqqqq] (6.777479066043686,-1.9749279121818235) circle (2pt);
\draw [fill=black] (7.400968867902419,-2.7567593946498534) circle (2pt);
\draw [fill=ffffff] (7.,-1.) circle (2pt);
\draw [fill=ffqqqq] (7.623489801858733,-1.78183148246803) circle (2pt);
\draw [fill=black] (8.623489801858733,-1.7818314824680301) circle (2pt);
\draw [fill=ffqqqq] (8.,-1.) circle (2pt);
\end{scriptsize}
\end{tikzpicture}

%% file: tikz/S71_v2.tex
\definecolor{ffffff}{rgb}{1.,1.,1.}
\definecolor{ffqqqq}{rgb}{1.,0.,0.}
\begin{tikzpicture}[line cap=round,line join=round,>=triangle 45,x=1.0cm,y=1.0cm]
\clip(-1.8,-2) rectangle (2,2);
\fill[line width=0.8pt,color=ffqqqq,fill=ffqqqq,fill opacity=0.25] (0.,0.) -- (1.,0.)  -- (0.623489801858733,1-0.2181685175319702) -- cycle;
\fill[line width=0.8pt,color=ffqqqq,fill=ffqqqq,fill opacity=0.25] (0.,0.) -- (0.623489801858733,1-0.2181685175319702)  -- (-1+0.777479066043686,1-0.02507208781817649) -- cycle;
\fill[line width=0.8pt,color=ffqqqq,fill=ffqqqq,fill opacity=0.25] (0.,0.) -- (-1+0.777479066043686,1-0.02507208781817649)  -- (-1+0.099031132097581,1-0.5661162608824418) -- cycle;
\fill[line width=0.8pt,color=ffqqqq,fill=ffqqqq,fill opacity=0.25] (0.,0.) -- (-1+0.099031132097581,1-0.5661162608824418) -- (-1+0.099031132097581,1-1.433883739117558) -- cycle;
\fill[line width=0.8pt,color=ffqqqq,fill=ffqqqq,fill opacity=0.25] (0.,0.) -- (-1+0.099031132097581,1-1.433883739117558) -- (-1+0.777479066043686,1-1.9749279121818235) -- cycle;
\fill[line width=0.8pt,color=ffqqqq,fill=ffqqqq,fill opacity=0.25] (0.,0.) -- (-1+0.777479066043686,1-1.9749279121818235) -- (0.623489801858733,1-1.78183148246803) -- cycle;
\fill[line width=0.8pt,color=ffqqqq,fill=ffqqqq,fill opacity=0.25] (0.,0.) -- (0.623489801858733,1-1.78183148246803) --  (1.,0.) -- cycle;

\draw [line width=0.5pt] (1.,0.)-- (0.,0.);
\draw [line width=0.5pt] (0.623489801858733,1-0.2181685175319702) -- (0.,0.);
\draw [line width=0.5pt] (-1+0.777479066043686,1-0.02507208781817649)-- (0.,0.);
\draw [line width=0.5pt] (-1+0.099031132097581,1-0.5661162608824418)-- (0.,0.);
\draw [line width=0.5pt] (-1+0.099031132097581,1-1.433883739117558)-- (0.,0.);
\draw [line width=0.5pt] (-1+0.777479066043686,1-1.9749279121818235)-- (0.,0.);
\draw [line width=0.5pt] (0.623489801858733,1-1.78183148246803)-- (0.,0.);

\draw  [line width=0.5pt, color = blue] (0,0) -- ({1.623489801858733/2},{(1-0.21816851753197058)/2});
\draw  [line width=0.5pt, color = blue] (0,0) -- (0.400968867902419/2,1.756759394649853/2) circle (1.5pt);
\draw  [line width=0.5pt, color = blue] (0,0) --  ({(-7+5.876510198141267)/2},1.40881165129938135/2) circle (1.5pt);
\draw  [line width=0.5pt, color = blue] (0,0) --  ({(-7+5.198062264195162)/2},0.) circle (1.5pt);
\draw  [line width=0.5pt, color = blue] (0,0) --  ({(-7+5.876510198141267)/2},-1.4088116512993816/2) circle (1.5pt);
\draw  [line width=0.5pt, color = blue] (0,0) --  (0.400968867902419/2,-1.7567593946498534/2) circle (1.5pt);
\draw  [line width=0.5pt, color = blue] (0,0) --  (1.623489801858733/2,{(1-1.7818314824680301)/2}) circle (1.5pt);

\draw [line width=0.5pt, color = red] (1.,0.) -- (0.623489801858733,1-0.2181685175319702) -- (-1+0.777479066043686,1-0.02507208781817649) -- (-1+0.099031132097581,1-0.5661162608824418) -- (-1+0.099031132097581,1-1.433883739117558) -- (-1+0.777479066043686,1-1.9749279121818235) -- (0.623489801858733,1-1.78183148246803) -- (1.,0.);

\draw [line width=0.5pt, color = red] (0.93,0.15) -- (0.93,0.3);
\draw [line width=0.5pt, color = red] (0.93,0.15) -- (0.81,0.24);

\draw [line width=0.5pt, color = red] (0.69,0.645) -- (0.81,0.55);
\draw [line width=0.5pt, color = red] (0.69,0.645) -- (0.69,0.5);


%
%
%
%
%

\def\ra{15}
\def\len{0.86776747823}



\draw ({1.3*cos(0)},{1.3*sin(0)}) node {\scriptsize $A_0$};
\draw ({1.3*cos(360/7)},{1.3*sin(360/7)}) node {\scriptsize $A_1$};
\draw ({1.3*cos(360*2/7)},{1.3*sin(360*2/7)}) node {\scriptsize $A_2$};
\draw ({1.3*cos(360*3/7)},{1.3*sin(360*3/7)}) node {\scriptsize $A_3$};
\draw ({1.3*cos(360*4/7)},{1.3*sin(360*4/7)}) node {\scriptsize $A_4$};
\draw ({1.3*cos(360*5/7)},{1.3*sin(360*5/7)}) node {\scriptsize $A_5$};
\draw ({1.3*cos(360*6/7)},{1.3*sin(360*6/7)}) node {\scriptsize $A_6$};

\draw ({1.3*cos(0)+0.5*cos(9*180/14)},{1.3*sin(0)+0.5*sin(9*180/14)}) node {\scriptsize $W_0$};
\draw ({1.3*cos(360/7)+0.5*cos(360/7+9*180/14)},{1.3*sin(360/7)+0.5*sin(360/7+9*180/14)}) node {\scriptsize $W_1$};
\draw ({1.3*cos(360*2/7)+0.5*cos(2*360/7+9*180/14)},{1.3*sin(360*2/7)+0.5*sin(2*360/7+9*180/14)}) node {\scriptsize $W_2$};
\draw ({1.3*cos(360*3/7)+0.5*cos(3*360/7+9*180/14)},{1.3*sin(360*3/7)+0.5*sin(3*360/7+9*180/14)}) node {\scriptsize $W_3$};
\draw ({1.3*cos(360*4/7)+0.5*cos(4*360/7+9*180/14)},{1.3*sin(360*4/7)+0.5*sin(4*360/7+9*180/14)}) node {\scriptsize $W_4$};
\draw ({1.3*cos(360*5/7)+0.5*cos(5*360/7+9*180/14)},{1.3*sin(360*5/7)+0.5*sin(5*360/7+9*180/14)}) node {\scriptsize $W_5$};
\draw ({1.3*cos(360*6/7)+0.5*cos(6*360/7+9*180/14)},{1.3*sin(360*6/7)+0.5*sin(6*360/7+9*180/14)}) node {\scriptsize $W_6$};


\begin{scriptsize}
\draw [fill=ffffff] (0.,0.) circle (2pt);

\draw [fill=blue] ({1.623489801858733/2},{(1-0.21816851753197058)/2}) circle (1.5pt);
\draw [fill=blue] (0.400968867902419/2,1.756759394649853/2) circle (1.5pt);
\draw [fill=blue] ({(-7+5.876510198141267)/2},1.40881165129938135/2) circle (1.5pt);
\draw [fill=blue] ({(-7+5.198062264195162)/2},0.) circle (1.5pt);
\draw [fill=blue] ({(-7+5.876510198141267)/2},-1.4088116512993816/2) circle (1.5pt);
\draw [fill=blue] (0.400968867902419/2,-1.7567593946498534/2) circle (1.5pt);
\draw [fill=blue] (1.623489801858733/2,{(1-1.7818314824680301)/2}) circle (1.5pt);

\draw [fill=ffqqqq] (0.623489801858733,1-0.2181685175319702) circle (2pt);
\draw [fill=ffqqqq] (-1+0.777479066043686,1-0.02507208781817649) circle (2pt);
\draw [fill=ffqqqq] (-1+0.099031132097581,1-0.5661162608824418) circle (2pt);
\draw [fill=ffqqqq] (-1+0.099031132097581,1-1.433883739117558) circle (2pt);
\draw [fill=ffqqqq] (-1+0.777479066043686,1-1.9749279121818235) circle (2pt);
\draw [fill=ffqqqq] (0.623489801858733,1-1.78183148246803) circle (2pt);
\draw [fill=ffqqqq] (1.,0.) circle (2pt);
\end{scriptsize}
\end{tikzpicture}

%% file: tikz/sphere_triang.tex
\definecolor{ffqqqq}{rgb}{1.,0.,0.}
\definecolor{qqqqff}{rgb}{0.,0.,1.}
\definecolor{ffffff}{rgb}{1.,1.,1.}
\begin{tikzpicture}[line cap=round,line join=round,>=triangle 45,x=1.0cm,y=1.0cm,
background rectangle/.style={fill=red!25}, show background rectangle]
\clip(-7.,-7.) rectangle (8.,7.);
\draw [line width=2pt,color=blue] (0.,0.)-- (2.,0.);
\draw [line width=2pt,color=blue] (0.,0.)-- (1.2469796037174672,1.5636629649360596);
\draw [line width=2pt,color=blue] (0.,0.)-- (-0.44504186791262856,1.9498558243636475);
\draw [line width=2pt,color=blue] (0.,0.)-- (-1.8019377358048383,0.8677674782351167);
\draw [line width=2pt,color=blue] (0.,0.)-- (-1.8019377358048385,-0.867767478235116);
\draw [line width=2pt,color=blue] (0.,0.)-- (-0.44504186791262923,-1.9498558243636475);
\draw [line width=2pt,color=blue] (0.,0.)-- (1.246979603717467,-1.5636629649360603);
\draw [line width=2pt,color=red] (2.,0.)-- (8.,0.);
\draw [line width=2pt,color=red] (1.2469796037174672,1.5636629649360596)-- (1.2*4.987918414869869,1.2*6.254651859744238);
\draw [line width=2pt,color=red] (-0.44504186791262856,1.9498558243636475)-- (-1.7801674716505143,7.79942329745459);
\draw [line width=2pt,color=red] (-1.8019377358048383,0.8677674782351167)-- (-7.207750943219353,3.4710699129404667);
\draw [line width=2pt,color=red] (-1.8019377358048385,-0.867767478235116)-- (-7.207750943219354,-3.471069912940464);
\draw [line width=2pt,color=red] (-0.44504186791262923,-1.9498558243636475)-- (-1.780167471650517,-7.79942329745459);
\draw [line width=2pt,color=red] (1.246979603717467,-1.5636629649360603)-- (1.2*4.987918414869868,-1.2*6.254651859744241);
\draw [line width=2pt] (0.,0.)-- (8.108719811121773,3.904953652058023);
\draw [line width=2pt] (0.,0.)-- (2.002688405606831,8.774351209636414);
\draw [line width=2pt] (0.,0.)-- (-5.611408216728602,7.036483342212271);
\draw [line width=2pt] (0.,0.)-- (-9.,0.);
\draw [line width=2pt] (0.,0.)-- (-5.611408216728605,-7.036483342212269);
\draw [line width=2pt] (0.,0.)-- (2.002688405606828,-8.774351209636416);
\draw [line width=2pt] (0.,0.)-- (8.108719811121773,-3.9049536520580275);
\begin{scriptsize}
\draw [fill=ffffff] (0.,0.) circle (5pt);
\draw [fill=qqqqff] (2.,0.) circle (5pt);
\draw [fill=qqqqff] (1.2469796037174672,1.5636629649360596) circle (5pt);
\draw [fill=qqqqff] (-0.44504186791262856,1.9498558243636475) circle (5pt);
\draw [fill=qqqqff] (-1.8019377358048383,0.8677674782351167) circle (5pt);
\draw [fill=qqqqff] (-1.8019377358048385,-0.867767478235116) circle (5pt);
\draw [fill=qqqqff] (-0.44504186791262923,-1.9498558243636475) circle (5pt);
\draw [fill=qqqqff] (1.246979603717467,-1.5636629649360603) circle (5pt);

\draw ({2*cos(0)+cos(30)},{2*sin(0)+sin(30)}) node {\LARGE $W_0$};
\draw ({2*cos(360/7)+cos(360/7+40)},{2*sin(360/7)+sin(360/7+40)}) node {\LARGE $W_1$};
\draw ({2*cos(360*2/7)+cos(360*2/7+30)},{2*sin(360*2/7)+sin(360*2/7+30)}) node {\LARGE $W_2$};
\draw ({2*cos(360*3/7)+cos(360*3/7+30)},{2*sin(360*3/7)+sin(360*3/7+30)}) node {\LARGE $W_3$};
\draw ({2*cos(360*4/7)+cos(360*4/7+30)},{2*sin(360*4/7)+sin(360*4/7+30)}) node {\LARGE $W_4$};
\draw ({2*cos(360*5/7)+cos(360*5/7+30)},{2*sin(360*5/7)+sin(360*5/7+30)}) node {\LARGE $W_5$};
\draw ({2*cos(360*6/7)+cos(360*6/7+35)},{2*sin(360*6/7)+sin(360*6/7+35)}) node {\LARGE $W_6$};

\end{scriptsize}
\end{tikzpicture}

%% file: tikz/S71_v3.tex
\definecolor{ffffff}{rgb}{1.,1.,1.}
\definecolor{ffqqqq}{rgb}{1.,0.,0.}
\begin{tikzpicture}[line cap=round,line join=round,>=triangle 45,x=1.0cm,y=1.0cm]
\clip(-0.7,-1) rectangle (2,1.8);
\fill[line width=0.8pt,color=ffqqqq,fill=ffqqqq,fill opacity=0.25] (0.,0.) -- (1.,0.)  -- (0.623489801858733,1-0.2181685175319702) -- cycle;
\draw [line width=0.5pt] (1.,0.)-- (0.,0.);
\draw [line width=0.5pt] (0.623489801858733,1-0.2181685175319702) -- (0.,0.);

\draw [line width=0.5pt, color = red] (1.,0.) -- (0.623489801858733,1-0.2181685175319702) ;

%
%
%
%
%

\def\ra{15}
\def\len{0.86776747823}

\draw (-0.1,0.2) node {\scriptsize $O$};


\draw ({1.3*cos(0)},{1.3*sin(0)}) node {\scriptsize $A$};
\draw ({1.3*cos(360/7)},{1.3*sin(360/7)}) node {\scriptsize $A$};

\draw ({1.3*cos(0)+0.5*cos(9*180/14)+0.05},{1.3*sin(0)+0.5*sin(9*180/14)}) node {\scriptsize $W$};


\draw [line width=0.5pt, color = red] (0.93,0.15) -- (0.93,0.3);
\draw [line width=0.5pt, color = red] (0.93,0.15) -- (0.81,0.24);

\draw [line width=0.5pt, color = red] (0.69,0.645) -- (0.81,0.55);
\draw [line width=0.5pt, color = red] (0.69,0.645) -- (0.69,0.5);

\draw [line width=0.5pt] (0.6,0.) -- (0.5,0.1);
\draw [line width=0.5pt] (0.6,0.) -- (0.5,-0.1);
\draw [line width=0.5pt] (0.5+0.05,0.) -- (0.5-0.1+0.05,0.1);
\draw [line width=0.5pt] (0.5+0.05,0.) -- (0.5-0.1+0.05,-0.1);

\draw [line width=0.5pt, color = blue] (0,0) -- ({1.623489801858733/2},{(1-0.21816851753197058)/2});

\draw[rotate=360/7] [line width=0.5pt] (0.6,0.) -- (0.5,0.1);
\draw[rotate=360/7]  [line width=0.5pt] (0.6,0.) -- (0.5,-0.1);
\draw[rotate=360/7]  [line width=0.5pt] (0.5+0.05,0.) -- (0.5-0.1+0.05,0.1);
\draw[rotate=360/7]  [line width=0.5pt] (0.5+0.05,0.) -- (0.5-0.1+0.05,-0.1);

\begin{scriptsize}
\draw [fill=ffffff] (0.,0.) circle (2pt);

\draw [fill=blue] ({1.623489801858733/2},{(1-0.21816851753197058)/2}) circle (1.5pt);

\draw [fill=ffqqqq] (0.623489801858733,1-0.2181685175319702) circle (2pt);
\draw [fill=ffqqqq] (1.,0.) circle (2pt);
\end{scriptsize}
\end{tikzpicture}

%% file: tikz/biliard_triang.tex
\definecolor{ffqqqq}{rgb}{1.,0.,0.}
\definecolor{qqqqff}{rgb}{0.,0.,1.}
\definecolor{ffffff}{rgb}{1.,1.,1.}
\begin{tikzpicture}[line cap=round,line join=round,>=triangle 45,x=1.0cm,y=1.0cm,
background rectangle/.style={fill=red!25}, show background rectangle]
\clip(-5.,-5.) rectangle (6.,5.);
\draw [line width=2pt,color=blue] (0.,0.)-- (2.,0.);

\draw [line width=2pt,color=red] (2.,0.)-- (8.,0.);

\draw [line width=2pt] (0.,0.)-- (-9.,0.);

\begin{scriptsize}
\draw [fill=ffffff] (0.,0.) circle (5pt);
\draw [fill=qqqqff] (2.,0.) circle (5pt);

\draw ({cos(30)-0.35},{sin(30)}) node {\LARGE $O$};

\draw ({2*cos(0)-0.25+cos(30)},{2*sin(0)+sin(30)}) node {\LARGE $W$};

\end{scriptsize}
\end{tikzpicture}

%% file: main.bbl
\providecommand{\bysame}{\leavevmode\hbox to3em{\hrulefill}\thinspace}
\providecommand{\MR}{\relax\ifhmode\unskip\space\fi MR }
\providecommand{\MRhref}[2]{%
  \href{http://www.ams.org/mathscinet-getitem?mr=#1}{#2}
}
\providecommand{\href}[2]{#2}
\begin{thebibliography}{1}

\bibitem{FMZ}
Giovanni Forni, Carlos Matheus, and Anton Zorich, \emph{Square-tiled cyclic
  covers}, J. Mod. Dyn. \textbf{5} (2011), no.~2, 285--318.

\bibitem{kucharczyk}
Robert~A Kucharczyk, \emph{Real multiplication on jacobian varieties}, arXiv
  preprint arXiv:1201.1570 (2012).

\bibitem{Mtab}
Howard Masur and Serge Tabachnikov, \emph{Rational billiards and flat
  structures}, Handbook of dynamical systems, {V}ol. 1{A}, North-Holland,
  Amsterdam, 2002, pp.~1015--1089.

\bibitem{MP}
Nicolas M{\"u}ller and Richard Pink, \emph{Hyperelliptic curves with many
  automorphisms}, arXiv preprint arXiv:1711.06599 (2017).

\bibitem{R}
Harry~E. Rauch, \emph{Theta constants on a {R}iemann surface with many
  automorphisms}, Symposia {M}athematica, {V}ol. {III} ({INDAM}, {R}ome,
  1968/69), Academic Press, London, 1970, pp.~305--323.

\bibitem{W}
J\"{u}rgen Wolfart, \emph{The ``obvious'' part of {B}elyi's theorem and
  {R}iemann surfaces with many automorphisms}, Geometric {G}alois actions, 1,
  London Math. Soc. Lecture Note Ser., vol. 242, Cambridge Univ. Press,
  Cambridge, 1997, pp.~97--112.

\bibitem{Z}
A.~{Zorich}, \emph{Flat surfaces}, Frontiers in number theory, physics, and
  geometry. {I}, Springer, Berlin, 2006, pp.~437--583. \MR{2261104}

\end{thebibliography}
